\documentclass[11pt]{amsart}
\usepackage[utf8]{inputenc}
 \usepackage[russian]{babel}
 \usepackage{amsmath,amssymb,url}
 
 \theoremstyle{plain}
 \newtheorem{theorem}{Теорема}
 \newtheorem{lemma}{Лемма}
 \newtheorem{corollary}{Следствие}
 \newtoks\thehProclaim

 \theoremstyle{definition}

\theoremstyle{plain}
\newtoks\thehProclaim
\newtheorem*{Proclaim}{\the\thehProclaim}
\newenvironment{proclaim}[1]{\thehProclaim{#1}\begin{Proclaim}}{\end{Proclaim}}

\theoremstyle{definition}
\newtoks{\thehRemark}
\newtheorem*{Remark}{\the\thehRemark}

\begin{document}

 \begin{center}
    {\Large\bf О финитной отделимости конечно порожденных ассоциативных колец}
 \end{center}
 \begin{center}\vspace{2mm}
     {\bf\small\copyright\ \ Станислав Кублановский}
 \end{center}\vspace{2mm}

    {\footnotesize Установлены необходимые и достаточные условия финитной отделимости моногенных колец. В качестве следствия доказано, что конечно порожденное PI-кольцо без кручения является финитно отделимым в том и только в том случае, когда его аддитивная группа конечно порождена.}

 \address{ТПО « Северный Очаг»\\
 Россия}
 \email{stas1107@mail. ru}

\section{Введение}
Понятие финитной аппроксимируемости и финитной отделимости в алгебраических системах вызывает постоянный интерес исследователей. Одной из причин этого интереса является связь с алгоритмическими проблемами. На эту связь указал еще академик А. И. Мальцев в работе 1958 \cite{mal}, а именно: в финитно аппроксимируемых (финитно отделимых) конечно определенных системах разрешима проблема равенства (проблема вхождения). Все это относится и к традиционным алгебраическим системам: полугруппам, группам и кольцам.

Напомним, что алгебраическая система $A$ называется финитно отделимой если для любого ее элемента $a$ и для любой подсистемы $ A^{ '}$ такой, что $a\not\in A'$, существует конечная система $F$ гомоморфизм $\varphi:A\rightarrow F$ такой, что $\varphi \left({a}\right)\not\in \varphi \left({A'}\right)$

Алгебраическая система $A$ называется финитно аппроксимируемой, если для любых ее двух различных элементов $a, b$ существует конечная система $F$ и гомоморфизм $\varphi:A\rightarrow F$ такой, что $\varphi \left({a}\right)\ne \varphi \left({b}\right)$.

Свойство финитной отделимости хорошо изучено в группах и полугруппах. Но для колец многие вопросы в этой тематике еще далеки от разрешения. Из кольцевых результатов в этой области стоит отметить работу О. Б. Пайсон, М. В. Волкова, М. В. Сапира 1999 \cite{pai}, в которой, в частности, описываются многообразия ассоциативных колец, все кольца которых (все конечно порожденные кольца) финитно отделимы.

Одним из открытых принципиальных вопросов здесь оставался критерий финитной отделимости моногенных колец (напомним, что моногенным называют кольцо, порожденное одним элементом). Разрешение этого вопроса являлось бы отправной точкой для исследований. Оказалось, что в отличие от групп и полугрупп, в кольцах дело обстоит значительно сложнее. Все изложенное в настоящей работе посвящено доказательству сформулированного автором здесь критерия.
\begin{theorem}
1. Для того, чтобы моногенное кольцо $Z\left<{\left. {a}\right>}\right. $ было финитно отделимым необходимо и достаточно, чтобы для некоторых целых чисел \\ $k, n, k_{1}, k_{2}, \dots, k_{n-1}$ ($n>0$) выполнялось равенство \begin{equation*}
 k\left({a^{n}+k_{1}a^{n-1}+k_{2}a^{n-2} + \dots +k_{n-1}a}\right)=0,
\end{equation*} причем $k$ --- положительное целое свободное от квадратов (то есть $k$ --- произведение различных простых чисел или $k=1$ ).

2. Пусть моногенное кольцо $Z\left<{\left. {a}\right>}\right. $ задано конечным набором определяющих соотношений $f_{1}\left({a}\right)=0, f_{2}\left({a}\right)=0, \dots, f_{m}\left({a}\right)=0$, где $f_{i}$ --- многочлены от одного переменного с целыми коэффициентами без свободных членов. Для того, чтобы моногенное кольцо $Z\left<{\left. {a}\right>}\right. $ было финитно отделимым, необходимо и достаточно, чтобы выполнялось два условия:

(i) Н.О.Д коэффициентов всех многочленов $f_{i}$ (вместе взятых) был свободен от квадратов;

(ii) унитарный многочлен, являющийся Н.О.Дом всех многочленов $f_{i}$ над полем рациональных чисел, имеет целые коэффициенты.
\end{theorem}

\begin{proclaim}{Замечание 1}
 Из теоремы Гильберта о базисе следует, что каждое конечно порожденное коммутативное кольцо является конечно определенным, то есть может быть задано конечным набором определяющих соотношений. Поэтому настоящая теорема дает полное, алгоритмически проверяемое, описание моногенных колец со свойством финитной отделимости.
\end{proclaim}

\begin{corollary}

Конечно порожденное PI-кольцо без кручения является финитно отделимым в том и только в том случае, когда его аддитивная группа конечно порождена.
\end{corollary}

\section{Определения и обозначения}

Определение группы, абелевой группы, кольца и поля, идеала, подкольца, фактор-кольца и канонического гомоморфизма из кольца в фактор-кольцо предполагаются известными. Рассматриваются кольца без требования существования 1, если не оговорено противное. Через $Z$ обозначается кольцо целых чисел.

{\bf 1.}
Элемент $a$ кольца $K$ называется алгебраическим, если он в этом кольце является корнем некоторого ненулевого многочлена с целыми коэффициентами (без свободного члена). Элемент, не являющийся алгебраическим, называется трансцендентным. Кольцо называется алгебраическим, если каждый его элемент алгебраический.

{\bf 2.}
Элемент $a$ кольца $K$ называется целым алгебраическим, если он в этом кольце является корнем некоторого ненулевого многочлена с целыми коэффициентами (без свободного члена) со старшим коэффициентом равным единице (такие многочлены называют унитарными).

{\bf 3.}
Элемент $a$ кольца $K$ назовем целокрутящимся, если он в этом кольце является корнем многочлена, являющегося произведением положительного целого числа на унитарный многочлен. Кольцо называется кольцом целого кручения, если каждый его элемент целокрутящийся.

{\bf 4.}
Через $D\left|{a}\right|$ обозначим наименьшую из возможных степеней ненулевых многочленов $f\left({x}\right)$ с целыми коэффициентами, для которых $f\left({a}\right)=0$ в кольце $K$. Такое число называют алгебраической степенью элемента $a$ в кольце $K$ (или просто --- степенью, когда это не вызывает конфликта).\\
Заметим, если $a$ --- алгебраический элемент, то его алгебраическая степень конечна. Алгебраическая степень трансцендентного элемента $a$ считается бесконечной.

{\bf 5.}
 Многочлены $f\left({x}\right)$ наименьшей степени с целыми коэффициентами (без свободного члена) и наименьшим положительным старшим членом, для которых $f\left({a}\right)=0$, называют минимальными многочленами алгебраического элемента $a$.

{\bf 6.}
 Целым кручением элемента $a$ назовем наименьшее из возможных натуральных чисел $k$ для которых выполнено равенство $k\cdot f\left({a}\right)=0$, где $f\left({x}\right)$ --- какой-либо унитарный многочлен с целыми коэффициентами (без свободного члена). Целое кручение элемента $a$ обозначим через $\tau _{a}$

{\bf 7.}
Наименьшую из возможных степеней унитарных многочленов $f\left({x}\right)$ без свободного члена), для которых выполняется равенство $k\cdot f\left({a}\right)=0$ (где $k$ --- какое-либо натуральное число) назовем показателем кручения элемента $a$ в кольце $K$ и обозначим $E\left|{a}\right|$.

\begin{proclaim}{Замечание 2}
 Если элемент $a$ не является целокрутящимся, то считаем $E\left|{a}\right|=\infty $ и $\tau _{a}=\infty $.
\end{proclaim}

{\bf 8.}
Если $a$ элемент кольца $K$, то через $I\left({a}\right)$ обозначают идеал кольца $K$, порожденный $a$.

{\bf 9.}
 Если $A$ и $B$ --- подмножества кольца $K$, то обозначают подмножества $A+B=\left\{{x+y\mid x\in A;y\in B}\right\};$ $A\cdot B=\left\{{xy \mid x\in A;y\in B}\right\}$. Если $L$ --- некоторое подмножество кольца целых чисел $Z$, $A$ --- подмножество кольца $K$, то обозначают подмножество кольца $K$: $L\cdot A=\left\{{xy \mid x\in L;y\in A}\right\}$.

\begin{flushleft}
Обозначения сохраняются в силе, если какое-либо из перечисленных выше подмножеств одноэлементно.
\end{flushleft}

{\bf 10.}
 Через $Z\left<{a_{1}, a_{1}, \dots a_{n}}\right>$ и $I\left({a_{1}, a_{1}, \dots a_{n}}\right)$ обозначается соответственно подкольцо и идеал кольца $K$, порожденные множеством элементов $a_{1}, a_{1}, \dots a_{n}\in K$.

{\bf 11.}
 Идеал, порожденный одним элементом, называется главным идеалом. Кольцо, порожденное одним элементом, называется моногенным.

{\bf 12.}
 Натуральное число $k$ называют свободным от квадратов, если оно не делится на квадрат простого числа. Легко видеть, что число свободное от квадратов это в точности произведение различных простых чисел, либо 1.

{\bf 13.}
 Пусть $n$ --- некоторое натуральное число. Через $I_{n}$ обозначаем идеал в кольце $K$, определяемый равенством $I_{n}=\{u\in K \mid nu=0\}$. Такие идеалы называют идеалами кручения.

{\bf 14.}
 Кольцо $R$ называется нетеровым, если выполнено одно из свойств:

1) каждый его идеал порождается (как идеал) конечным числом элементов;

2) любая возрастающая цепочка его идеалов $T_{1}\subset T_{2}. . $ стабилизируется, то есть найдется такое $n$, что $T_{n}=T_{n+1}=T_{n+2}=\dots $ ;

3) В каждом совокупности идеалов кольца есть максимальный элемент.\\
Легко доказывается, что указанные свойства эквивалентны (упражнение). В литературе в определении нетерового кольца иногда требуют наличие в кольце единицы. Мы здесь это не требуем.
\vspace{2mm}

{\bf Примеры:} кольцо целых чисел, кольцо четных чисел, любое поле. \\
По теореме Гильберта кольцо многочленов над кольцом целых чисел от конечного числа переменных -нетерово.

{\bf 15.}
 Кольцо $K$ называется кольцом главных идеалов, если каждый его идеал главный, то есть порождается (как идеал) одним элементом. Кольцо главных идеалов может не иметь единицы. Например, подкольцо кольца многочленов $Z_{p}[x]$ над простым полем $Z_{p}$ (поле остатков по простому модулю $p$), порожденное переменной $x$. Это все многочлены без свободного члена. Такое кольцо обозначим $Z^{* }_{p}[x]$.
{\bf 16.}
Легко видеть, что в этом кольце возможно деление с остатком. Отсюда вытекает, что $Z^{* }_{p}[x]$ --- кольцо главных идеалов. 

\begin{flushleft}
{\bf Примеры:} кольцо целых чисел, любое поле.
\end{flushleft}

\begin{proclaim}{Замечание 3}
 Кольцо многочленов от одной переменной с целыми коэффициентами не является кольцом главных идеалов.
 \end{proclaim}
 
\begin{flushleft}
{\bf Свойства (упражнение):}
\end{flushleft}

1) Класс колец главных идеалов замкнут относительно взятия гомоморфных образов и не замкнут относительно взятия подколец.

2) Кольцо многочленов от одной переменной с коэффициентами из кольца главных идеалов может не быть кольцом главных идеалов. (Например, кольцо многочленов от одной переменной с целыми коэффициентами не является кольцом главных идеалов).

3) Каждое кольцо главных идеалов нетерово.

{\bf 17.}
 Пусть $R$ —-- кольцо (как правило, считающееся коммутативным и не обязательно с единичным элементом). $ R$-модулем называется абелева группа $M$ с операцией умножения на элементы кольца $ R$ : $R\times M\rightarrow M. \left({r, m}\right)\rightarrow rm, $ которая удовлетворяет следующим четырем условиям:\\
1) $\forall r_{1}, r_{2}\in R$ $\forall m\in M$ $\left({r_{1}r_{2}}\right)m=r_{1}\cdot \left({r_{2}m}\right)$, \\
2) $\forall m\in M$ $1\cdot m=m$, если в кольце $R$, есть $1$, \\
3) $\forall r\in R$, $\forall m_{1}, m_{2}\in M$ $r\left({m_{1}+m_{2}}\right)=rm_{1}+rm_{2}$, \\
4) $\forall r_{1}, r_{2}\in R$, $m\in M$ $\left({r_{1}+r_{2}}\right)m=r_{1}m+r_{2}m$.

\begin{proclaim}{Примечание 1}
 В случае некоммутативного кольца $R$ такие модули часто называются левыми. Правыми модулями называют в этом случае такие объекты, у которых умножение элементов абелевой группы $M$ на элементы кольца $R$ рассматривается с правой стороны.
\end{proclaim}

{\bf 18.} Если в кольце $R$ есть единица $1$ и выполнено условие:\\
5) $\forall m\in M$ $1\cdot m=m$, то такой $R$-модуль $M$ называют унитарным. \\
В дальнейшем, если не оговорено противное, мы будем здесь рассматривать унитарные модули.

{\bf 19.}
Если в модуле $M$ определено умножение так, что вместе со сложением $M$ является кольцом, и если выполнены условия:\\
6) $\forall r_{1}, r_{2}\in R, m\in M$ $\left({r_{1}r_{2}}\right)\cdot m=r_{1}\left({r_{2}m}\right)$ \\
7) $\forall r\in R$, $\forall m_{1}, m_{2}\in M$ $r\left({m_{1}\cdot m_{2}}\right)=\left({rm_{1}}\right)\cdot m_{2}$, \\
то модуль $M$ называется алгеброй над кольцом $R$.

{\bf 20.}
Пусть $M$ --- некоторый унитарный $R$-модуль. Для любого подмножества $S\subset M$ множество линейных комбинаций
\begin{center}

 $a_{1}u_{i_{1}}+\dots+a_{m}u_{i_{m}}, $ где $a_{1}, \dots,, a_{m}\in R$ и $u_{i_{1}}, \dots, u_{i_{m}}\in S$.
\end{center}
является наименьшим подмодулем в $M$, содержащим подмножество $S. $ Он называется подмодулем, порожденным множеством $S$, и обозначается через $[S]$. Если $[S]=M$, то говорят, что модуль $M$ порождается множеством $S$. Если множество $S$ можно выбрать конечным, то говорят, что $M$ конечно порожден.

{\bf 21.}
Модуль, порожденный одним элементом, называется циклическим.

{\bf 22.}
 Система $\left\{{u_{1}, \dots, u_{n}}\right\}$ элементов модуля $M$ называется линейно независимой, если $r_{1}u_{1}+ \dots +r_{n}u_{n}=0$ ($r_{i}\in R$) только при $r_{i}=0 $ ($i=1, \dots, n$)\\
Линейно независимая система, порождающая модуль $M$, называется базисом. Если равенство $r_{1}u_{1}+\cdot \cdot \cdot +r_{n}u_{n}=0$ возможно не только при нулевых коэффициентах $r_{i}\in R$, то система $\left\{{u_{1}, \dots, u_{n}}\right\}$ элементов модуля $M$ называется линейно зависимой.

{\bf 23.}
 Мы будем использовать теорему о линейной зависимости линейных комбинаций. Пусть $\left\{{u_{1}, \dots, u_{m}}\right\}$ и $\left\{{v_{1}, \dots, v_{k}}\right\}$ две совокупности элементов модуля $M$ над областью целостности и все элементы второй совокупности суть линейные комбинации элементов первой совокупности. Тогда, если $k>m$, то совокупность $\left\{{v_{1}, \dots, v_{k}}\right\}$ элементов модуля $M$ линейно зависима.

{\bf 24.}
 Конечнопорожденный модуль, обладающий базисом, называется свободным.

{\bf 25.}
 Количество элементов в базисе называется рангом свободного модуля.

\begin{proclaim}{Замечание 4}
 Из пп. 23 следует, что количество элементов в базисе свободного модуля над кольцом главных идеалов не зависит от выбора базиса, ранг свободного модуля определен однозначно.
\end{proclaim}

{\bf 26.}
 Сумма подмодулей $L_{1}+L_{2}+\dots L_{n}$ модуля $M$ называется прямой суммой, если каждый элемент этой суммы единственным образом представим как сумма
 \begin{center}
 $u_{1}+u_{2}+\dots, u_{n}$, где $u_{i}\in L_{i}$ ($i=1, \dots, n$).
 \end{center}
 Tо есть из равенства \begin{equation*}
 u_{1}+u_{2}+\dots+ u_{n}=u^{'}_{1}+u^{'}_{2}+\dots+ u^{'}_{n}
 \end{equation*} следуют равенства
\begin{equation*}
 u_{i}=u^{'}_{i} (u_{i}, u^{'}_{i}\in L_{i} (i=1, \dots, n).
\end{equation*}

{\bf 27.}
 Кольцо $R$ называется кольцом без делителей нуля, если произведение любых его ненулевых элементов является ненулевым элементом.

\begin{flushleft}
{\bf Примеры:} кольцо целых чисел, любое поле. \\
\end{flushleft}
{\bf Свойства (упражнение):}\\
- Класс колец без делителей нуля замкнут относительно взятия подколец и не замкнут относительно гомоморфизмов. \\
- Кольцо многочленов с коэффициентами из кольца без делителей нуля само является кольцом без делителей нуля. \\
- Если в $R$-модуле $M$ существуют линейно независимые системы элементов, то кольцо $R$ обязано быть без делителей нуля (верно и обратное).

{\bf 28.}
 Коммутативное кольцо $R$ с единицей без делителей нуля, являющееся кольцом главных идеалов, называют областью целостности.

{\bf 29.}
 Евклидово кольцо --- область целостности $R$, для которой определена евклидова функция $D$ (евклидова норма) из множества ненулевых элементов кольца в множество неотрицательных целых чисел, такая, что возможно деление с остатком по норме меньшим делителя, то есть для любых $ a, b\in R$, и $b\ne 0$ имеется представление $ a=bq+r$, для которых $q, r\in R$ $D\left({r}\right)<D\left({b}\right)\}$ или $r=0$.

\begin{flushleft}
{\bf Примеры:} кольцо целых чисел ($D\left({u}\right)=\left|{u}\right|$, любо поле ($D\left({u}\right)=1$), кольцо многочленов от одной переменной над полем ($D\left({u}\right)=deg\left({u}\right)$ --- степень многочлена $u$). \\
\end{flushleft}
{\bf Свойства:}\\
- каждое евклидово кольцо является кольцом главных идеалов;\\
- в каждом евклидовом кольце для любых двух элементов существует их Н.О.Д, который может быть найден с помощью последовательного деления с остатком (алгоритм Евклида);\\
- в каждом евклидовом кольце $R$ выполняется теорема Евклида о линейном представлении Н.О.Д: Для любых $a_{1}, \dots a_{n}\in R$ имеет место $x_{1}a_{1}+ \dots +x_{n}a_{n}=d$,
для некоторых $x_{1}, \dots x_{n}\in R$ и $d=$Н.О.Д $\left({a_{1}, \dots a_{n}}\right)$. \\
Мы будем использовать хорошо известные свойства модулей над областями целостности.

{\bf 30.}
 1) (Теорема о строении подмодулей свободного модуля) Всякий подмодуль $N$ свободного унитарного $R$-модуля $M$ ранга $n$ над областью целостности является свободным $R$--модулем ранга $m\le n$, причем существует такой базис $\left\{{e_{1}, e_{2}, \dots e_{n}}\right\}$ модуля $M$ и такие (ненулевые) элементы $r_{i}\in R$, ($i=1, 2\dots n$), что $\{r_{1}e_{1}, r_{2}e_{2}, \dots, r_{m}e_{m}\}$ и $r_{i}$ делит $r_{i+1}$ ($i=1, 2\dots, m-1$) --- базис подмодуля $N$. \\
2) Всякий унитарный конечно порожденный модуль над областью целостности является прямой суммой циклических подмодулей. \\
3) Всякий унитарный конечно порожденный модуль без кручения над областью целостности является свободным модулем конечного ранга (см. \cite{byrb}, Алгебра том 2, стр. 54, следствие 2 )

{\bf 31.}
 Алгебра А над коммутативным кольцом $R$ с единицей называется PI-алгеброй, если в ней выполнено нетривиальное тождество
\begin{equation*}
 \sum\limits_{k=1}^{n}{\lambda _{k}u_{k}}=0,
\end{equation*}
 где $u_{k}$ --- семейство попарно различных слов от некоммутирующих переменных, а $\lambda _{k}\in R$ и идеал, порожденный всеми коэффициентами $\lambda _{k}$ в кольце $R$, совпадает со всем кольцом $R$. Каждое кольцо можно рассматривать как алгебру над кольцом целых чисел. Поэтому понятие PI-кольца есть частный случай понятия PI-алгебры.

\begin{proclaim}{Замечание 5}
 Мы будем использовать следующее известное свойство PI-колец и PI-алгебр. Теорема А. И. Ширшова о высоте \cite{shir}. Для любой конечно порожденной ассоциативной PI-алгебры $A$ над коммутативным кольцом найдутся натуральное число $h$ и элементы $a_{1}, a_{2}, \dots, a_{n}\in A$ такие, что любой элемент алгебры $A$ представляется в виде линейной комбинации элементов вида $w=a^{\alpha _{1}}_{i_{1}}a^{\alpha _{2}}_{i_{2}}\dots, a^{\alpha _{k}}_{i_{k}}$, для некоторого натурального числа $k<h$ (число $k$ называют высотой слова $w$ над $a_{1}, a_{2}, \dots, a_{n}$.
\end{proclaim}
\begin{proclaim}{Замечание 6}
 Мы будем использовать следующую теорему А. И. Мальцева, доказанную им в работе \cite{mal}: прямое произведение конечного набора финитно отделимых колец является финитно отделимым кольцом.
\end{proclaim}

\section{Вспомогательные утверждения}

Среди доказанных в этом параграфе утверждений часть представляет известные результаты. Но для удобства чтения и независимости от ссылок автор приводит их доказательства в учебных целях.
\begin{lemma}
1. Класс финитно отделимых колец замкнут относительно подколец и гомоморфных образов. \\
2. Любое конечное подмножество элементов финитно отделимого кольца можно отделить от непересекающегося с ним подкольца некоторым гомоморфизмом исходного кольца в конечное кольцо.
\end{lemma}
\begin{proof}
1. Пусть $K$ --- финитно отделимое кольцо и $K/I$ --- гомоморфный образ $K$ по некоторому идеалу $I$. Пусть $a\in K$, $A$ --- подкольцо кольца $K$. Через $\overline{a}$ и $\overline{A}$ обозначим образ элемента $a$ и образ подкольца $A$ соответственно при каноническом гомоморфизме $K\rightarrow K/I$. Каждый элемент и каждое подкольцо кольца $K/I$ есть образ некоторого элемента $a$ и подкольца $A$ кольца $K$. Пусть $\overline{a}\not\in \overline{A}$ в кольце $K/I$. Это означает, что $a\not\in A+I$ в кольце $K$. Заметим, что $A+I$ --- это подкольцо в кольце $K$. В силу финитной отделимости $K$ существует конечное кольцо $F$ и гомоморфизм $\varphi:K\rightarrow F$ такой, что $\varphi \left({a}\right)\not\in \varphi \left({A+I}\right)$. Пусть $I'- $ ядро этого гомоморфизма $\varphi $. Без ограничения общности (по теореме о гомоморфизмах колец) будем считать, что $F=K/I'$ и $\varphi $ --- канонический гомоморфизм $K\rightarrow K/I'$. Ясно, что $K/\left({I+I'}\right)$ --- конечное кольцо (как гомоморфный образ конечного кольца $F=K/I'$). Имеем $a\not\in A+I+I'$ в кольце $K$. Это означает, что при каноническом гомоморфизме $K/I\rightarrow K/\left({I+I'}\right)$ элемент $\overline{a}$ отделяется от подкольца $\overline{A}$. Вывод: $K/I$ --- финитно отделимое кольцо.

Остальные утверждения {\bf леммы 1} предоставим доказать читателю как упражнения.
\end{proof}
\begin{lemma}
 1. Финитно отделимое кольцо является алгебраическим.\\
2. В финитно отделимом кольце для любых двух коммутирующих элементов $a, b$ и для любого простого числа $p$ имеет место: $f\left({a, b}\right)=0$ для некоторого многочлена $f\left({x, y}\right)$ с целыми коэффициентами (без свободного члена), среди которых есть не делящиеся на $p$.
\end{lemma}
\begin{proclaim}{Замечание 7}
Легко видеть, что в кольце многочленов $R[x, y]$ над любым полем $R$ равенство $f\left({a, b}\right)=0$ для переменных $x, y$ не выполняется и поэтому указанное кольцо не является финитно отделимым.
\end{proclaim}
\begin{proof}
1. Пусть $K$ --- финитно отделимое кольцо и $a\in K$. Покажем, что $a$ --- алгебраический элемент. Без ограничения общности, на основании {\bf леммы 1}, далее будем считать, что $K$ совпадает с моногенным кольцом $Z\left<{a}\right>$. Пусть $p$ --- произвольное простoе числo. Пусть $B=\{a^{2n}-a^{n}+a \mid n=1, 2, 3,\dots \}$, $[B]$ --- аддитивная подгруппа кольца $K$, порожденная множеством $B$. Рассмотрим подмножество $K'$ кольца $K$, определяемое равенством: $K'=p[B]+p^{2}K$. Покажем, что, что $K'$ --- подкольцо в кольце $K$. Замнутость $K'$ по сложению (вычитанию) следует из определений. Покажем замкнутость по умножению. Действительно, для любых элементов $u, v\in K'$ имеет место $u\cdot v\in p^{2}K\subset K'$. Если $pa\not\in K'$, то на основании финитной отделимости кольца $K$ существует конечное кольцо $F$ и гомоморфизм $\varphi:K\rightarrow F$ такой, что $\varphi \left({pa}\right)\not\in \varphi \left({K'}\right)$. Но в конечном кольце моногенная мультипликативная полугруппа, порожденная любым элементом --- конечна. В конечной полугруппе для любого элемента некоторая его степень является идемпотентом, то есть $\varphi \left({a}\right)^{2n}=\varphi \left({a}\right)^{n}$ для некоторого натурального $n$. Тогда\begin{equation*}
    \varphi \left({pa}\right)=p\varphi \left({a^{2n}-a^{n}+a}\right)\in \varphi \left({pB}\right)\subset \varphi \left({K'}\right).
\end{equation*} Получаем противоречие из предположения $pa\not\in K'$. Вывод: $pa\in K'$. Отсюда следует равенство\begin{equation*}
     pa=pz_{n}\left({a^{2n}-a^{n}+a}\right)+\dots+ pz_{1}\left({a^{2}-a+a}\right)+p^{2}g\left({a}\right)
\end{equation*} для некоторого многочлена $g\left({x}\right)$ с целыми коэффициентами (без свободного члена), некоторого натурального числа $n$ и некоторых целых чисел $z_{1}, \dotsб z_{n}$. Это означает, что $f\left({a}\right)=0$, где\begin{equation*}
    f\left({x}\right)=pz_{n}\left({x^{2n}-x^{n}+x}\right)+ \dots + pz_{1}\left({x^{2}-x+x}\right)+p^{2}g\left({x}\right)-px.
\end{equation*} Заметим, что не все коэффициенты многочлена $f\left({x}\right)$ делятся на $p^{2}$, то есть $f\left({x}\right)$ --- ненулевой многочлен (без свободного члена). Вывод: $a$ --- алгебраический элемент. Утверждение 1 леммы 2 доказано.\vspace{1mm}

2. Пусть $K$ --- финитно отделимое кольцо и $a, b\in K$. Далее, без ограничения общности, на основании леммы 1 считаем, что кольцо $K$ совпадает с подкольцом $Z\left<{a, b}\right>$. Пусть $p$ --- произвольное простoе числo. Пусть $B=\{a^{2n}-a^{n}+a\mid n=1, 2, 3\dots \}$, $[B]$ --- аддитивная подгруппа кольца $K$, порожденная множеством $B$. Рассмотрим подмножество $K'$ кольца $K$ определяемое равенством: $K'=b[B]+b^{2}K$. Покажем, что, что $K'$ --- подкольцо в кольце $K$. Замнутость $K'$ по сложению (вычитанию) следует из определений. Покажем замкнутость по умножению. Действительно, для любых элементов
$u, v\in K'$ имеет место $u\cdot v\in b^{2}K\subset K'$. Если $ba\not\in K'$, то на основании финитной отделимости кольца $K$ существует конечное кольцо $F$ и гомоморфизм $\varphi:K\rightarrow F$ такой, что $\varphi \left({ba}\right)\not\in \varphi \left({K'}\right)$.
Но в конечном кольце моногенная мультипликативная полугруппа, порожденная любым элементом --- конечна. В конечной полугруппе для любого элемента некоторая его степень является идемпотентом, то есть $\varphi \left({a}\right)^{2n}=\varphi \left({a}\right)^{n}$ для некоторого натурального $n$. Тогда\begin{equation*}
     \varphi \left({ba}\right)=\varphi \left({b}\right)\varphi \left({a^{2n}-a^{n}+a}\right)\in \varphi \left({pB}\right)\subset \varphi \left({K'}\right).
\end{equation*} Получаем противоречие из предположения $ba\not\in K'$. Вывод: $ba\in K'$. Отсюда следует равенство\begin{equation*}
    ba=bz_{n}\left({a^{2n}-a^{n}+a}\right)+\dots+bz_{1}\left({a^{2}-a+a}\right)+b^{2}g\left({a, b}\right)
\end{equation*} для некоторого многочлена $g\left({x, y}\right)$ с целыми коэффициентами (без свободного члена), некоторого натурального числа $n$ и некоторых целых чисел $z_{1}, \dots, z_{n}$. Это означает, что $f\left({a, b}\right)=0$, где \begin{equation*}
    f\left({x, y}\right)=yz_{n}\left({x^{2n}-x^{n}+x}\right)+\dots + yz_{1}\left({x^{2}-x+x}\right)+y^{2}g\left({x, y}\right)-yx.
\end{equation*}
Заметим, что не все коэффициенты многочлена $f\left({x, y}\right)$ делятся на $p$. Действительно, если все числа $z_{1}, \dots z_{n}$ делятся на $p$, то коэффициент у одночлена $yx$ в многочлене $f\left({x, y}\right)$ не делятся на $p$. Если же не все $z_{1}, \dots, z_{n}$ делятся на $p$, то среди этих чисел найдется $z_{m}$ с наибольшим номером $m$, которое не делится на $p$. Тогда коэффициент у одночлена $yx^{2m}$ в многочлене $f\left({x, y}\right)$ не делится на $p$. Вывод: свойство 2 выполнено. Лемма 2 доказана.
\end{proof}
\begin{proclaim}{Замечание 8}
Во второй части леммы 2 (свойство 2) фактически доказано, что для любого простого числа $p$ и для любых двух элементов $a, b$ финитно отделимого кольца имеет место $b\cdot f\left({a}\right)=b^{2}g\left({a, b}\right)$ для некоторых многочленов $f\left({x}\right)$ (без свободного члена) и $g\left({x, y}\right)$ с целыми коэффициентами и без свободного члена, причем не все коэффициенты многочлена $f\left({x}\right)$ делятся на $p$. Откуда следует \begin{equation*}
    f_{0}\left({b}\right)a^{n}+f_{1}\left({b}\right)a^{n-1}+\dots f_{n-1}\left({b}\right)a=0
\end{equation*} для некоторого натурального числа $n$ и некоторых многочленов \\ $f_{0}\left({x}\right), f_{1}\left({x}\right), +\dots f_{n-1}\left({x}\right)$ с целыми коэффициентами без свободных членов, причем не все коэффициенты многочлена $f_{0}\left({x}\right)$ делятся на $p$.
\end{proclaim}

\begin{proclaim}{Предложение 1}
1) Для каждого алгебраического элемента $a$ любого кольца минимальный многочлен $f\left({x}\right)$ -единственный. Каждый многочлен $g\left({x}\right)$ с целыми коэффициентами (и без свободного члена), для которого $g\left({a}\right)=0$, при умножении на некоторое число $k>0$ делится на минимальный многочлен $f\left({x}\right)$ элемента $a$ в кольце многочленов с целыми коэффициентами, причем все коэффициенты $f\left({x}\right)$ делятся на $k$. \\
2) Если $l\cdot \widetilde{f}\left({a}\right)=0$ для некоторого унитарного многочлена $\widetilde{f}\left({x}\right)$ с целыми коэффициентами без свободного члена и некоторого целого положительного числа $l$, то минимальный многочлен $f\left({x}\right)$ элемента $a$ имеет вид $f\left({x}\right)=d\cdot f^{* }\left({x}\right)$, где $f^{* }\left({x}\right)$ -некоторый унитарный многочлен с целыми коэффициентами и без свободного члена.
\end{proclaim}
\begin{proof}
1) Пусть $f\left({x}\right)$ --- минимальный многочлен элемента $a$ кольца $K$ и пусть $g\left({x}\right)$ --- многочлен с целыми коэффициентами, для которого $g\left({a}\right)=0$.
В кольце $Q[x]$ многочленов с рациональными коэффициентами можно разделить многочлен $g\left({x}\right)$ на $f\left({x}\right)$ с остатком. То есть представить $g\left({x}\right)$ в виде $g\left({x}\right)=f\left({x}\right)\alpha \left({x}\right)+\beta \left({x}\right)$, где $\alpha \left({x}\right), \beta \left({x}\right)$  --- некоторые многочлены с рациональными коэффициентами, причем степень многочлена $\beta \left({x}\right)$ меньше степени многочлена $f\left({x}\right)$. Существует такое натуральное число $k$, что $k\alpha \left({x}\right)$ и $k\beta \left({x}\right)$ - многочлены с целыми коэффициентами. (достаточно в качестве $k$ взять нок знаменателей коэффициентов многочленов $\alpha \left({x}\right), \beta \left({x}\right)$). Тогда имеет место
\begin{center}
     $kg\left({x}\right)=f\left({x}\right)\left({k\alpha \left({x}\right)}\right)+k\beta \left({x}\right)$,
\end{center} откуда
\begin{center}
    $k\beta \left({a}\right)=kg\left({a}\right)-f\left({a}\right)\left({k\alpha \left({a}\right)}\right)=0$.
\end{center} Многочлен $k\beta \left({x}\right)$ - с целыми коэффициентами и имеет степень такую же как и многочлен $\beta \left({x}\right)$, то есть меньшую степени минимального многочлена $f\left({x}\right)$. Если $k\beta \left({x}\right)$ ненулевой многочлен, то это противоречит определению минимального многочлена. Вывод: $k\beta \left({x}\right)=0$.\\
То есть $kg\left({x}\right)=f\left({x}\right)\left({k\alpha \left({x}\right)}\right)$. Из последнего равенства вытекает единственность минимального многочлена.

Действительно, если $g\left({x}\right)$ --- еще один минимальный многочлен, то, как показано выше, в кольце многочленов с рациональными коэффициентами $f\left({x}\right)$ и $g\left({x}\right)$ должны делиться друг на друга. Это возможно для многочленов одинаковой степени только при условии, что они отличаются на постоянный множитель. То есть $g\left({x}\right)=r\cdot f\left({x}\right)$ Но, поскольку старшие коэффициенты многочленов $f\left({x}\right)$ и $g\left({x}\right)$ равны, заключаем $r=1$, откуда $g\left({x}\right)=f\left({x}\right)$.

Теперь докажем вторую часть первого пункта предложения 1. Пусть теперь $g\left({x}\right)$  ---  многочлен с целыми коэффициентами (без свободного члена), для которого $g\left({a}\right)=0$ и не обязательно минимальный. Далее можно использовать следующие весьма очевидные свойства многочленов с целыми коэффициентами. Если Н.О.Д коэффициентов многочлена назвать его содержанием, то выполняются следующие свойства:\\
а) содержание произведения двух многочленов равно произведению их содержаний. \\
б) если один многочлен делится на другой, то содержание первого делится на содержание второго. \\
Выше было установлено, что в кольце многочленов с целыми коэффициентами многочлен $kg\left({x}\right)$ делится на многочлен $f\left({x}\right)$ для некоторого натурального числа $k$. Будем считать $k$ -наименьшим из возможных с этим свойством. Пусть $kg\left({x}\right)=f\left({x}\right)h\left({x}\right)$ для некоторого многочлена $h\left({x}\right)$ с целыми коэффициентами. Если содержание многочлена $h\left({x}\right)$ не взаимно просто с $k$, то в последнем равенстве можно было бы поделит на их общий делитель и получилось бы равенство $k'g\left({x}\right)=f\left({x}\right)h'\left({x}\right)$ многочленов с целыми коэффициентами, у которого $k'<k$. это противоречит выбору $k$. Полученное означает, что $d$  --- содержание $f\left({x}\right)$ должно делиться на $k$. Утверждение 1 настоящего предложения доказано.\\
2) Пусть $l\cdot \widetilde{f}\left({a}\right)=0$ для некоторого унитарного многочлена $\widetilde{f}\left({x}\right)$ с целыми коэффициентами без свободного члена и некоторого целого положительного числа $l$.

По первому пункту настоящего предложения $k\cdot l\cdot \widetilde{f}\left({x}\right)=f\left({x}\right)h\left({x}\right)$, где $k$ некоторое положительное целое число, $f\left({x}\right)$ -минимальный многочлен элемента $a$ в кольце $K$, а $h\left({x}\right)$  --- многочлен с целыми коэффциентами, причем содержание многочлена $h\left({x}\right)$ взаимно просто с числом $k$. и все коэффициенты многочлена $f\left({x}\right)$ делятся на $k$. Многочлен $f\left({x}\right)$ можно представить в виде $f\left({x}\right)=d\cdot f^{* }\left({x}\right)$, где $f^{* }\left({x}\right)$  ---  многочлен с целыми коэффициентами с единичным содержанием и без свободного члена, а $d$  --- содержание многочлена $f\left({x}\right)$. (говоря проще, из многочлена $f\left({x}\right)$ вынесем наибольший общий делитель его коэффициентов). Имеем $k\cdot l\cdot \widetilde{f}\left({x}\right)=df^{* }\left({x}\right)h\left({x}\right)$, откуда $l\cdot \widetilde{f}\left({x}\right)=d'f^{* }\left({x}\right)h\left({x}\right)$, где $d=kd'$. Пусть $h\left({x}\right)=d''h^{* }\left({x}\right)$, где $d''$  --- содержание многочлена $h\left({x}\right)$, а $h^{* }\left({x}\right)$  ---  многочлен с целыми коэффициентами с единичным содержанием и без свободного члена (говоря проще, из многочлена $h\left({x}\right)$ вынесем наибольший общий делитель его коэффициентов). Тогда $l\cdot \widetilde{f}\left({x}\right)=d'd''f^{* }\left({x}\right)h^{* }\left({x}\right)$, откуда $l=d'd''$ (мы использовали: содержание произведения двух многочленов равно произведению их содержаний), то есть $\widetilde{f}\left({x}\right)=f^{* }\left({x}\right)h^{* }\left({x}\right)$. Поскольку $\widetilde{f}\left({x}\right)$  --- унитарный многочлен, то из последнегоравенства вытекает, что $f^{* }\left({x}\right)$  --- унитарный многочлен (достаточно приравнять старшие коэффициенты многочленов $\widetilde{f}\left({x}\right)$ и $f^{* }\left({x}\right)h^{* }\left({x}\right)$). Утверждение 2 настоящего предложения доказано. Предложение 1 доказано полностью.
\end{proof}
\begin{lemma}
Финитно отделимое кольцо является кольцом целого кручения, причем для любого элемента показатель кручения равен его алгебраической степени.
\end{lemma}
\begin{proof}
Пусть $K$  ---  финитно отделимое кольцо и $a\in K$. Покажем, что $a$  --- целокрутящийся элемент и его показатель кручения равен его алгебраической степени. Предположим противное. \\
{\bf По лемме 1} $a$  ---  алгебраический элемент и пусть его алгебраическая степень равна $n$, то есть $D\left|{a}\right|=n$ (см. пп. 4). По определению это означает, что для некоторого набора целых чисел $m_{0}, m_{1}, \dots, m_{n-1}$ выполняется равенство:
\begin{center}
 $m_{0}a^{n}+m_{1}a^{n-1}+\dots m_{n-1}a=0$, причем $m_{0}>0$.
\end{center}
Заметим, что $n>1$. В противном случае элемент $a$ был бы целокрутящимся и его показатель кручения был бы равен его алгебраической степени $1$, что противоречит предположению. Пусть $d=$Н.О.Д$\left({m_{0}, m_{1}, \dots, m_{n-1}}\right)$ и пусть $k_{i}=m_{i}:d\left({i=0, 1,\dots, n-1}\right)$. Тогда по определению Н.О.Д имеем Н.О.Д $\left({k_{0}, k_{1}, \dots, k_{n-1}}\right)=1$. Покажем, что $k_{0}=1$.

Предположим противное, то есть $k_{0}\ne 1$. Обозначим через $I_{d}=\{v\in K \mid dv=0\}$. Нетрудно видеть, что $I_{d}$  --- идеал в кольце $K$. По {\bf лемме 1} фактор кольцо $K/I_{d}$  --- финитно отделимо. Обозначим через $\overline{u}$ образ элемента $u\in K$ при каноническом гомоморфизме $K\rightarrow K/I_{d}$. На основании определений заключаем, что в $K/I_{d}$ имеет место:
\begin{equation*}
    k_{0}\overline{a}^{n}+k_{1}\overline{a}^{n-1}+\dots+ k_{n-1}\overline{a}=0
\end{equation*} и
\begin{equation}
    D\left|{\overline{a}}\right|=n \tag{*}.
\end{equation}

Тогда для некоторого натурального числа $m$, не превосходящего $n-1$, число $k_{m}$ не делится на некоторый простой делитель числа $k_{0}$ (в противном случае Н.О.Д чисел $k_{i}$ был бы отличен от 1). Далее будем считать, что $m$ выбрано наименьшим из возможных с указанным свойством. Перепишем равенство (*) в двух видах:
\begin{equation}
    k_{0}\overline{a}^{n}+k_{1}\overline{a}^{n-1}+\dots+ k_{m-1}\overline{a}^{n-m+1}=-k_{m}\overline{a}^{n-m}\dots -k_{n-1}\overline{a} \tag{1}.
\end{equation} и
\begin{equation}
     k_{0}\overline{a}^{n}=-k_{1}\overline{a}^{n-1}\dots -k_{n-1}\overline{a}  \tag{2}.
\end{equation}
Обозначим через $f\left({x}\right)$ и $\varphi \left({x}\right)$ следующие многочлены:
\begin{equation*}
    f\left({x}\right)= k_{0}x^{n}+k_{1}x^{n-1}+\dots+ k_{m-1}x^{n-m+1}
\end{equation*} и
 \begin{equation*}
      \varphi \left({x}\right)=-k_{m}x^{n-m}\dots -k_{n-1}x.
 \end{equation*}
Обозначим через $A$ подмножество кольца $K/I_{d}$ определяемое равенством:
\begin{equation*}
     A=Z\cdot \left\{{k_{0}\overline{a}}\right\}+Z\cdot \left\{{k^{2}_{0}\overline{a}^{2}}\right\}+\dots + Z\cdot \left\{{k^{n-1}_{0}\overline{a}^{n-1}}\right\}.
\end{equation*}

Ясно, что $A$ замкнуто по сложению (вычитанию), то есть $A$ аддитивная подгруппа кольца $K/I_{d}$. Из равенства $\left({2}\right)$ следует замкнутость $A$ по умножению. Чтобы в этом убедиться достаточно проверить $k_{0}\overline{a}\cdot k^{n-1}_{0}\overline{a}^{n-1}\in A$. Действительно,\begin{equation*}
    k_{0}\overline{a}\cdot k^{n-1}_{0}\overline{a}^{n-1}=k^{n-1}_{0}k_{0}\overline{a}^{n}=k^{n-1}_{0}\left({-k_{1}\overline{a}^{n-1}\dots -k_{n-1}\overline{a}}\right)\in A.
\end{equation*}
Следовательно, $A$  --- подкольцо кольца $K/I_{d}$. Заметим, что $A$ содержит все элементы вида $k^{s}_{0}\overline{a}^{s}$, для любого натурального $s$. Обозначим через $M$ подмножество кольца $K/I_{d}$ определяемое равенством:\begin{equation*}
    M=\{\overline{b}=z_{n-1}\overline{a}^{n-1}+\dots+ z_{i}\overline{a}^{i}+ \dots+ z_{1}\overline{a} \mid \left|{z_{i}}\right|<k^{i}_{0}\left({i=1, \dots, n-1}\right)\}
\end{equation*}
Отметим, что $M$  --- конечное множество. Пусть $M^{* }$  --- множество ненулевых элементов из $M$. Покажем, что пересечение $M^{* }\cap A=\left\{{0}\right\}$. Действительно, в противном случае получилось бы $D\left|{\overline{a}}\right|<n$, что противоречит $\left({* }\right)$. Поскольку кольцо $K/I_{d}$ финитно отделимо, то по лемме 1 существует конечное кольцо $\Omega $ и гомоморфизм $\chi:K/I_{d}\rightarrow \Omega $ такой, что $\chi \left({M^{* }}\right)\cap \chi \left({A}\right)=\emptyset $. Тогда в конечном кольце существуют натуральные числа $t$ и $h$ такие, что $\chi \left({k^{t}_{0}\overline{a}}\right)=\chi \left({k^{t+h}_{0}\overline{a}}\right)$. Из последнего равенства следует $\chi\left({k^{t}_{0}\overline{a}}\right)=\chi \left({k^{t+sh}_{0}\overline{a}}\right)$ для любого натурального числа $s$, откуда следует \begin{equation*}
    \chi \left({k^{t}_{0}\overline{a}^{s}}\right)=\chi \left({k^{t+sh}_{0}\overline{a}^{s}}\right)=\chi \left({k^{t+sh-s}_{0}k^{s}_{0}\overline{a}^{s}}\right)=k^{t+sh-s}_{0}\chi \left({k^{s}_{0}\overline{a}^{s}}\right)
\end{equation*}
Как отмечено выше, $k^{s}_{0}\overline{a}^{s}\in A$. Поэтому из последнего равенства вытекает \begin{equation*}
    \chi \left({k^{t}_{0}\overline{a}^{s}}\right)\in \chi \left({A}\right) \tag{3}
\end{equation*} для любого натурального числа $s$. Пусть $L=\{k_{0}, k_{1}, \dots, k_{m-1}\}$. Пусть $q$  --- некоторое натуральное число.
Обозначим через $Z\left({L^{q}}\right)$  --- идеал кольца целых чисел, порожденный множеством $L^{q}=L\cdot L\cdot L\dots\cdot L$ ( $q$ сомножителей, см. пп. 9). Для дальнейших рассуждений докажем следующую импликацию:
\begin{flushleft}
Если для некоторых целых чисел $c_{i }\in Z\left({L^{q}}\right)\left({i>n-m}\right)$ и для некоторого целого числа $z$ имеет место равенство:
\end{flushleft}

\begin{center}
    $\sum\limits_{i>n-m}{c_{i}\overline{a}^{i}}=k^{z}_{m}\varphi \left({\overline{a}}\right)+r\left({\overline{a}}\right)$
\end{center} для некоторого многочлена с целыми коэффициентами $r\left({x}\right)$ без свободного члена степени меньшей $n-m$, то имеет место аналогичное равенство:
\begin{center}
     $\sum\limits_{i>n-m}{c^{'}_{i}\overline{a}^{i}}=k^{z'}_{m}\varphi \left({\overline{a}}\right)+r'\left({\overline{a}}\right)$
\end{center}
 для некоторых целых чисел $c^{'}_{i }\in Z\left({L^{q+1}}\right)\left({i>n-m}\right)$ и для некоторого целого числа $z'$ и некоторого многочлена с целыми коэффициентами $r'\left({x}\right)$ без свободного члена степени меньшей $n-m$.
Для доказательства импликации заметим следующее:\\
а) Для некоторой достаточно большого натурального числа $l$ в кольце многочленов с целыми коэффициентами (от одного переменного) возможно деление с остатком.
\begin{center}
     $k^{l}_{m}x^{i}=U_{i}\left({x}\right)\varphi \left({x}\right)+R_{i}\left({x}\right)$ (для всех $i>n-m$, для которых $c_{i}\ne 0$)
\end{center}
По определению деления с остатком степень многочлена остатка $R_{i}\left({x}\right)$ меньше степени делителя, то есть меньше $n-m$. Отметим, что многочлен $R_{i}\left({x}\right)$ без свободного члена.\\
б) Из условия импликации следует путем умножения на $k^{l}_{m}$:
\begin{equation*}
     \sum\limits_{i>n-m}{c_{i}k^{l}_{m}\overline{a}^{i}}=k^{z+l}_{m}\varphi \left({\overline{a}}\right)+k^{l}_{m}r\left({\overline{a}}\right),
\end{equation*}
 откуда на основании (а) получаем:
 \begin{equation*}
      \sum\limits_{i>n-m}{c_{i}\left({U_{i}\left({\overline{a}}\right)\varphi \left({\overline{a}}\right)+R_{i}\left({\overline{a}}\right)}\right)}=k^{z+l}_{m}\varphi \left({\overline{a}}\right)+k^{l}_{m}r\left({\overline{a}}\right)
 \end{equation*}\\
в) Из равенства $\left({1}\right)$ следует $\varphi \left({\overline{a}}\right)=f\left({\overline{a}}\right)$. Поэтому заменяя в предыдущем равенстве $\varphi \left({\overline{a}}\right)$ на $f\left({\overline{a}}\right)$, получаем:
\begin{equation*}
     \sum\limits_{i>n-m}{c_{i}\left({U_{i}\left({\overline{a}}\right)f\left({\overline{a}}\right)+R_{i}\left({\overline{a}}\right)}\right)}=k^{z+l}_{m}\varphi \left({\overline{a}}\right)+k^{l}_{m}r\left({\overline{a}}\right)
\end{equation*}

Коэффициенты многочлена $f\left({x}\right)$ принадлежат множеству $L$ (по определению). Поэтому коэффициенты многочлена $c_{i}\left({U_{i}\left({x}\right)f\left({x}\right)}\right)$ принадлежат множеству $Z\left({L^{q+1}}\right)$, поскольку по условию все $c_{i }\in Z\left({L^{q}}\right)$. Чтобы доказать рассматриваемую импликацию, осталось в левой части последнего равенства раскрыть скобки и перенести все $c_{i}R_{i}\left({\overline{a}}\right)$ в правую часть. Итак, импликация доказана.

Из этой импликации методом математической индукции приходим к выводу: для любого натурального числа $q$ и для некоторых целых чисел $c_{i }\in Z\left({L^{q}}\right)\left({i>n-m}\right)$ и для некоторого целого числа $z$ имеет место равенство:
\begin{center}
    $\sum\limits_{i>n-m}{c_{i}\overline{a}^{i}}=k^{z}_{m}\varphi \left({\overline{a}}\right)+r\left({\overline{a}}\right)$
\end{center} для некоторого многочлена с целыми коэффициентами. $r\left({x}\right)$ без свободного члена степени меньшей $n-m$. Базой индукции является равенство $\left({1}\right)$. Заметим, что

\begin{center}
    $k^{z}_{m}\varphi \left({x}\right)+r\left({x}\right)=q_{1}x+q_{2}x^{2}+\dots+ q_{n-m}x^{n-m}$
\end{center} для некоторых целых чисел $q_{j}$ ($j=1,\dots, n-m$), причем $q_{n-m}=-k^{z+1}_{m}$. Поделим каждое число $q_{j}$ с остатком на число $k^{j}_{0}$ и получим равенства: $q_{j}=p_{j}k^{j}_{0}+q^{* }_{j}$, где $0\le q^{* }_{j}<k^{j}_{0}$ ($j=1, \dots, n-m$). Заметим, что $ q^{* }_{n-m}>0$, поскольку $k^{z+1}_{m}$ не делится на $k_{0}$.\\ Имеем \begin{equation*}
    k^{z}_{m}\varphi\left({\overline{a}}\right)+r\left({\overline{a}}\right)=\sum\limits_{j=1}^{n-m}{\left({p_{j}k^{j}_{0}+q^{* }_{j}}\right)\overline{a}}^{j} =\sum\limits_{j=1}^{n-m}{p_{j}k^{j}_{0}\overline{a}}^{j}+\sum\limits_{j=1}^{n-m}{q^{* }_{j}\overline{a}}^{j}
\end{equation*} \\ Получаем: \begin{equation*}
    \chi \left({\sum\limits_{i>n-m}{c_{i}\overline{a}^{i}}}\right)=\chi \left({k^{z}_{m}\varphi \left({\overline{a}}\right)+r\left({\overline{a}}\right)}\right)
\end{equation*}
Последнее равенство перепишем следующим образом:
\begin{equation}
    \chi \left({\sum\limits_{j=1}^{n-m}{p_{j}k^{j}_{0}\overline{a}}^{j}}\right)+\chi \left({\sum\limits_{j=1}^{n-m}{q^{* }_{j}\overline{a}}^{j}}\right)   \tag{4}
\end{equation}
Для достаточно большого числа $q$ любое число $c_{i}$ из идеала $Z\left({L^{q}}\right)$ делится на $k^{t}_{0}$. Поэтому на основании $\left({3}\right)$ заключаем $\chi \left({\sum\limits_{i>n-m}{c_{i}\overline{a}^{i}}}\right)\in \chi \left({A}\right)$.

На основании определения $A$ имеем $\left({\sum\limits_{j=1}^{n-m}{p_{j}k^{j}_{0}\overline{a}}^{j}}\right)\in A$.

Поэтому из равенства $\left({4}\right)$ получаем $\chi \left({\sum\limits_{j=1}^{n-m}{q^{* }_{j}\overline{a}}^{j}}\right)\in A$.

На основании определения множества $M$ имеем $\sum\limits_{j=1}^{n-m}{q^{* }_{j}\overline{a}}^{j}\in M$.

Поскольку $\chi \left({M^{* }}\right)\cap \chi \left({A}\right)=\emptyset $, то заключаем $\sum\limits_{j=1}^{n-m}{q^{* }_{j}\overline{a}}^{j}=0$.

Но тогда получилось бы $D\left|{\overline{a}}\right|\le n-m<n$, что противоречит $\left({* }\right)$.

Противоречие получилось из предположения $k_{0}\ne 1$. Это означает, что $k_{0}=1$, то есть $a$  --- целокрутящийся элемент и его показатель кручения равен его алгебраической степени. Лемма 3 доказана.
\end{proof}

\begin{proclaim}{Замечание 9}
Мы доказали, что в кольце $K$ выполнено равенство \\ $d\cdot \left({a^{n}+k_{1}a^{n-1}+\dots k_{n-1}a}\right)=0$.
\end{proclaim}

\begin{lemma}
Каждый элемент финитно отделимого кольца имеет конечное целое кручение, свободное от квадратов.
\end{lemma}

\begin{proof}
 Пусть $K$ финитно отделимое кольцо, $a\in K$. Тогда по лемме 3 элемент $a$  целокрутящийся, то есть его целое кручение $\tau _{a}$ конечно (см. пп. 6). Покажем, что $\tau _{a}$  --- свободное от квадратов. Предположим противное. Тогда $\tau _{a}=p^{2}q$ для некоторого простого числа $p$ и некоторого натурального числа $q$. По определению целого кручения (см. пп. 6) $p^{2}q\cdot f\left({a}\right)=0$, где $f\left({x}\right)$  --- некоторый унитарный многочлен с целыми коэффициентами. Через $\overline{a}$ обозначим образ элемента $a$ при каноническом гомоморфизме $K\rightarrow K/I_{q}$, где $I_{q}$  --- идеал кручения (см. пп. 13).

Далее обозначим через $y$ элемент кольца $K/I_{q}$, определяемый равенством $y=f\left({\overline{a}}\right)$. Заметим, что в кольце $K/I_{q}$ выполнено равенство $p^{2}y=0$, поскольку $p^{2}f\left({a}\right)\in I_{q}$. Пусть $n$  --- какое-либо натуральное число. Через $y_{n}$ обозначим элемент кольца $K/I_{q}$, определяемый равенством \begin{center}
    $y_{n}=py+p\left({y^{2n}-y^{n}}\right)$.
\end{center}

Пусть $A$  --- аддитивная подгруппа кольца $K/I_{q}$, порожденная семейством элементов $y_{n}$. Нетрудно видеть, что $A$  --- замкнуто по умножению, поскольку произведение любых элементов из $A$ равно $0$ (это следует из равенства $p^{2}y=0$). Вывод: $A$  --- подкольцо кольца $K/I_{q}$. Далее возможны два случая: $py\not\in A$ или $py\in A$.

Если $py\not\in A$, то в силу финитной отделимости кольца $K/I_{q}$ по лемме 1 имеем $\varphi \left({py}\right)\not\in \varphi \left({A}\right)$ для некоторого конечного кольца $F$ и гомоморфизма $\varphi:K/I_{q}\rightarrow F$. Но в конечном кольце для каждого элемента некоторая степень является идемпотентом. Поэтому $\varphi \left({y}\right)^{2n}=\varphi \left({y}\right)^{n}$ для некоторого $n$. Это означает, что $\varphi \left({y^{2n}-y^{n}}\right)=0$, откуда следует $\varphi \left({y_{n}}\right)=\varphi \left({py}\right)$, то есть $\varphi \left({py}\right)\in \varphi \left({A}\right)$. Противоречие. Это означает, что случай $py\not\in A$ не имеет места.

Поэтому $py\in A$, откуда $py=\sum\limits_{i}{k_{i}y_{i}}$ для некоторого семейства целых чисел $k_{i}$. Из этой суммы можно отбросить слагаемые, в которых $k_{i}$ делятся на $p$ (поскольку, как отмечено выше, $p^{2}y=0$, что влечет $py_{i}=0$).

Пусть $k_{n}$ ненулевой коэффициент в рассматриваемой сумме с наибольшим номером $n$ Поскольку $k_{n}$ не делится на простое число $p$, то Н.О.Д $\left({k_{n}, p}\right)=1$. По теореме Евклида о наибольшем общем делителе выполняется равенство $k_{n}\cdot l+p\cdot z=1$ для некоторых целых чисел $l, z$. Тогда из равенства $py=\sum\limits_{i}{k_{i}y_{i}}$ путем умножения на $l$ получаем:\begin{center}
    $lpy=\sum\limits_{i<n}{lk_{i}y_{i}}+lk_{n}y_{n}$,
\end{center}  откуда следует:
\begin{center}
     $\left({1-pz}\right)y_{n}+\sum\limits_{i<n}{lk_{i}y_{i}}-lpy=0$.
\end{center}
Из последнего равенства получаем:
\begin{center}
     $y_{n}+\sum\limits_{i<n}{lk_{i}y_{i}}-lpy=0$,
\end{center} поскольку, как отмечалось выше, $py_{n}=0$. Это означает, что \begin{center}
   $ p\left({y^{2n}-y^{n}+\sum\limits_{i<n}{lk_{i}\left({y+y^{2i}-y^{i}}\right)}-ly+y}\right)=0$
\end{center} 
Обозначим через $g\left({x}\right)$ унитарный многочлен со старшим членом $x^{2n}$: \begin{center}
    $g\left({x}\right)=x^{2n}-x^{n}+\sum\limits_{i<n}{lk_{i}\left({x+x^{2i}-x^{i}}\right)}-lx+x$.
\end{center} Из последнего равенства с участием $y$ следует $pg\left({f\left({\overline{a}}\right)}\right)=0$ в кольце $K/I_{q}$, то есть $p\cdot gf\left({\overline{a}}\right)=0$. Пусть $h\left({x}\right)=g\left({f\left({x}\right)}\right)$. Ясно, что $h\left({x}\right)$  --- унитарный многочлен (как композиция унитарных). Имеем $ph\left({\overline{a}}\right)=0$ в кольце $K/I_{q}$, откуда
$qph\left({a}\right)=0$ в кольце $K$. Полученное означает, что целое кручение элемента $a $ не превосходит $pq$. Но это противоречит равенству $\tau _{a}=p^{2}q$. Противоречие получено из предположения о том, что кручение элемента $a$ делится на квадрат простого числа. Вывод: кручение элемента $a$  --- свободно от квадратов. Лемма 4 доказана.
\end{proof}
\begin{lemma}
 Если аддитивная группа кольца конечно порождена, то это кольцо финитно отделимо.
\end{lemma}
\begin{proof}
Пусть аддитивная группа кольца $K$ конечно порождена. Пусть $a\in K$, $A$  ---  подкольцо и $a\not\in A$. Обозначим через $\widetilde{a}$ образ элемента $a$ при каноническом гомоморфизме аддитивных абелевых групп $K\rightarrow K/A$. По условию $\widetilde{a}\ne 0$ в группе $K/A$ и эта группа конечно порождена (как гомоморфный образ конечнопорожденной). Как известно, конечно порожденная абелева группа является прямой суммой конечного числа циклических групп. Отсюда вытекает, что для любого ненулевого элемента конечно порожденной абелевой группы существует натуральное число на который этот элемент не делится. Итак, существует натуральное число $q$ для которого $\widetilde{a}\not\in q\cdot K/A$. Это означает, что в $K$ имеет место $a\not\in qK+A$. Заметим, что $qK$  --- двусторонний идеал в кольце $K$. Пусть $f:K\rightarrow K/qK$ канонический гомоморфизм колец, Из $a\not\in qK+A$ вытекает $f\left({a}\right)\not\in f\left({A}\right)$ Осталось заметить, что конечно порожденная периодическая абелева группа является конечной, в частности, $K/qK$  --- конечное множество. Лемма 5 доказана.
\end{proof}
\begin{lemma}
Если $a$  --- целый алгебраический элемент кольца $K$, то моногенное кольцо $Z\left<{a}\right>$ финитно отделимо.
\end{lemma}
 \begin{proof}
  Достаточно заметить, что если $a$  --- целый алгебраический элемент кольца, то аддитивная группа кольца $Z\left<{a}\right>$ конечно порождена. Тогда можно применить лемму 5. Лемма 6 доказана.
 \end{proof}
\begin{proclaim}{Предложение 2}
 Пусть в конечно порожденном кольце $K$ есть идеал $I$, такой, что аддитивная фактор-группа $K/I$ конечно порождена. Тогда $I$ является конечно порожденным кольцом.
\end{proclaim}

\begin{proof}
Пусть в конечно порожденном кольце $K$ есть идеал $I$, такой, что аддитивная фактор-группа $K/I$ конечно порождена. Пусть $\overline{a}$  --- образ $a\in K$ при каноническом гомоморфизме $K\rightarrow K/I$, $\overline{e_{i}}$ ($i=1, \dots, n, e_{i}\in K$) --- система образующих абелевой группы $K/I$. Из определений вытекает, что\begin{center}
    $K=\sum\limits_{i=1}^{n}{Z}\cdot e_{i}+I$.
\end{center}  Пусть $\{a_{1}, \dots a_{m}\}$  ---  система образующих кольца $K$. Заметим, что

(i) $e_{i}\cdot e_{j}=\left({\sum\limits_{k=1}^{n}{z_{i, j, k}e_{k}}}\right)+b_{i, j}$ для некоторых $z_{i, j, k}\in Z$ и $b_{i, j}\in I$ .

(ii) $a_{q}=\left({\sum\limits_{k=1}^{n}{z_{i}e_{i}}}\right)+b_{q}$ для некоторых $z_{k}\in Z$ и $b_{k}\in I$, $q=1, 2, \dots, m$. \\
Пусть $I^{* }$  --- подкольцо кольца $K$, порожденное конечным семейством элементов:
\begin{center}
     $M=\{b_{q}, b_{i, j}, e_{i}b_{q}, b_{q}e_{i}, e_{k}b_{i, j}, b_{i, j}e_{k} \mid 1\le i, j, k\le n, 1\le q\le m\}. $
\end{center}
Поскольку $I$  --- идеал в $K$, то $I^{* }\subset I$. Нетрудно заметить, что $a_{i}I^{* }\subset I^{* }$ и $I^{* }a_{i}\subset I^{* }$, то есть $I^{* }$  --- идеал в кольце $K$. Из определений вытекает, что $K=\sum\limits_{i=1}^{n}{Z}\cdot e_{i}+I^{* }$. Это означает, что аддитивная фактор-группа $K/I^{* }$ конечно порождена. Тогда конечно порождена аддитивная фактор-группа $I/I^{* }$ (как подгруппа конечно порожденной коммутативной $K/I^{* }$). Пусть $\widetilde{i_{1}}, \widetilde{i_{2}}, \dots, \widetilde{i_{l}}$  --- конечная система образующих аддитивной группы $I/I^{* }$ (где $\widetilde{i}$  --- образ элемента $i\in I$ при каноническом гомоморфизме), тогда семейство $i_{1}, \dots, i_{l}$ вместе с конечным семейством $M$ образующих кольца $I^{* }$ является порождающим множеством кольца $I$. Предложение 2 доказано.
\end{proof}
\begin{proclaim}{Предложение 3}
 Пусть в конечно порожденном кольце $K$ есть идеал $I$, такой, что аддитивная группа фактор-кольца $K/I$ конечно порождена. Пусть идеал $I$ сам является финитно отделимым кольцом. Тогда кольцо $K$  финитно отделимо.
\end{proclaim}
\begin{proof}
Пусть аддитивная группа кольца $K/I$ конечно порождена и $I$  --- идеал кольца $K$, сам являющийся финитно отделимым кольцом. Пусть $a\in K$ и $a\not\in A$, где $A$  --- некоторое подкольцо кольца $K$.

Пусть $a\in I$ Тогда $a\not\in A\cap I$. Поскольку $I$ --- финитно отделимое кольцо и $A\cap I$  --- подкольцо в $I$, то существует идеал $I'$ в кольце $I$ такой, что фактор $I/I'$  --- конечное кольцо и $\overline{a}\not\in \overline{A\cap I}$, где $\overline{a}$  --- образ $a$ при каноническом гомоморфизме $I\rightarrow I/I'$. Это означает $a\notin \left({A\cap I}\right)+I'$ .

Пусть $I^{*}$ --- идеал кольца $K$, порожденный множеством $II'I$. Заметим, что $II'I\subset I'$, то есть $I^{* }\subset I'$. По предложению 2 $I$  --- конечно порожденное кольцо. Тогда можно применить предложение 2 к кольцу $I$ и сделать вывод, что $I'$  ---  конечно порожденное кольцо. Заметим, что фактор--кольцо $I'/I^{* }$ является кольцом с нулевым умножением любых трех элементов и очевидно, что аддитивная группа конечно порожденного кольца с таким свойством конечно порождена. Расширение конечно порожденной группы при помощи конечно порожденной само является конечно порожденной группой. Поэтому аддитивная группа кольца $K/I^{* }$ конечно порождена.

Поскольку, как отмечено выше, $a\not\in \left({A\cap I}\right)+I'$, то $a\not\in \left({A\cap I}\right)+I^{* }$ (так как $I^{* }\subset I'$). Отсюда следует, что $a\notin A+I^{* }$. Действительно, в противном случае $a=b+i^{* }$ для некоторых $b\in A$, $i^{* }\in I^{* }\subset I$. Тогда $b=a-i^{* }\in I$, поскольку $a\in I$ (по условию рассматриваемого случая) и $i^{* }\in I^{* }\subset I$. Получаем $b\in A\cap I$, откуда $a\in \left({A\cap I}\right)+I^{* }$. Противоречие. Итак, $a\notin A+I^{* }$.

Поэтому $\widetilde{a}\notin \widetilde{A}$, где $\widetilde{a}$  --- образ $a$ при каноническом гомоморфизме $\theta :K\rightarrow K/I^{* }$. По лемме 5 кольцо $K/I^{* }$  финитно отделимо. Это означает, что элемент $\widetilde{a}$ можно отделить от подкольца $\widetilde{A}$ некоторым гомоморфизмом $f$ из $K/I^{* }$ в некоторое конечное кольцо $F$. Композиция канонического гомоморфизма $\theta $ и гомоморфизма $f$, то есть $\varphi =f\circ \theta:K\rightarrow F$ отделяет элемент $a$ от подкольца $A$, то есть $\varphi \left({a}\right)\not\in \varphi \left({A}\right)$.

Теперь рассмотрим второй возможный случай $a\notin I$. Если $a\notin A+I$, то повторя рассуждения, приведенные выше, получим, $\varphi \left({a}\right)\not\in \varphi \left({A}\right)$ для некоторого гомоморфизма $\varphi:K\rightarrow F$, где $F$  --- некоторое конечное кольцо.

Предположим $a\in A+I$. Тогда $a=b+i$ для некоторых $b\in A$, $i\in I$. Тогда $i\notin A$ (в противном случае получилось бы $a=b+i\in A$  --- противоречие).

По случаю, разобранному выше, существует такой идеал $I^{* }$ в кольце $K$, такой, что $I^{* }\subset I$ и аддитивная группа кольца $K/I^{* }$ конечно порождена. и $i\not\in A+I^{* }$.

Заметим, что $i=a-b$ и $b\in A+I^{* }$ (поскольку $b\in A$) Если бы $a\in A+I^{* }$, тогда и $a-b\in A+I^{* }$ (поскольку сумма подгрупп является подгруппой), то есть $i=a-b\in A+I^{* }$  --- противоречие. Следовательно, $a\not\in A+I^{* }$. С этого момента, повторяя рассуждения предыдущего случая (при котором $a\in I$), приходим к выводу, что  $\varphi \left({a}\right)\not\in \varphi \left({A}\right)$ для некоторого гомоморфизма $\varphi:K\rightarrow F$, где $F$  --- некоторое конечное кольцо. Предложение 3 доказано.
\end{proof}
\begin{proclaim}{Предложение 4}
Пусть $R$  ---  евклидово кольцо. Всякий подмодуль $N$ свободного $R$-модуля $M$ ранга $n$ является свободным $R$-модулем ранга $m\le n$, причем существуют такой базис $\left\{{e_{1}, e_{2}, \dots, e_{n}}\right\}$ модуля $M$ и такие \textup(ненулевые\textup) элементы $r_{i}\in R$ $(i=1, 2\dots n)$, что $\{r_{1}e_{1}, r_{2}e_{2}, \dots, r_{m}e_{m}\}$ --- базис подмодуля $N$ и $r_{i}$ делит $r_{i+1}$ ($i=1, 2\dots m-1$).
\end{proclaim}
\begin{proof}
Проведем методом математической индукции относительно ранга $n$. \\
1) Пусть $n=1$. Это означает, что $M=R\cdot e_{1}$. Тогда $N=I\cdot e_{1}$ для некоторого подмножества $I\subset R$. Очевидно, $I$  --- идеал в кольце $R$. Поскольку $R$  --- кольцо главных идеалов. то $I=\left({r_{1}}\right)$ для некоторого $r_{1}\in R$. Тогда $\{r_{1}e_{1}\}$  --- базис подмодуля $N$.

2) Пусть $n>1$ Предположим, что утверждение теоремы верно для свободных модулей рангов меньших $n$.

3) Зафиксируем произвольный базис $\left\{{e_{1}, e_{2}, \dots e_{n}}\right\}$ в модуле $M$. Каждый элемент $u$ модуля $M$ представим единственным образом в виде линейной комбинации элементов базиса $\left\{{e_{1}, e_{2}, \dots e_{n}}\right\}$, то есть \begin{center}
$u=a_{1, u}e_{1}+\cdot \cdot \cdot +a_{n, u}e_{n} (a_{i, u}\in R) $
\end{center}. Обозначим через $N_{i}=\{a_{i, u}\mid u\in N\}$. Возможно некоторые множества $N_{i}$ будут состоять из одного $0$. Без ограничения общности будем считать, что это происходит начиная с некоторого номера. То есть для некоторого $m\le n$ множества $N_{i}$ отличны от нуля, а $N_{k}=\{0\}$ для всех $k$ из промежутка $m<k\le n$. Если $m<n$, то $N$ содержится в свободном модуле с базисом $\left\{{e_{1}, e_{2}, \dots e_{m}}\right\}$. Ранг этого модуля (равен $m$) меньше $n$ и можно воспользоваться индукционным предположением.

Далее будем считать, что $m=n$, то есть все $N_{i}$ отличны от нуля. Из определения вытекает, что $N_{i}$  --- идеал в кольце $R$. Поскольку $R$  --- кольцо главных идеалов, то $N_{i}=\left({r_{i}}\right)$ для некоторого $r_{i}\in R$ и $r_{i}\ne 0$ ($i=1, \dots k$). Среди семейства $r_{i}$ (рассматриваемых по всевозможным базисам $\left\{{e_{1}, e_{2}, \dots e_{n}}\right\}$ в модуле $M$) есть элемент с наименьшей евклидовой нормой.

Без ограничения общности, будем считать, что $\left\{{e_{1}, e_{2}, \dots, e_{n}}\right\}$ --- такой базис и $r_{1}$  --- такой элемент. Из определений следует существование элемента $u\in N$ такого, что
\begin{equation*}
    u=r_{1}e_{1}+a_{2, u}e_{2}+\cdots+a_{n, u}e_{n}
\end{equation*}  
для некоторых $a_{i, u}\in R;$ $i=2, \dots, n$. \\
Пусть $N'=N\cap [e_{2}, \dots e_{n}]$. Из определений следует, что $[e_{2}, \dots, e_{n}]$  --- свободный модуль ранга $n-1$. По индукционному предположению в модуле $[e_{2}, \dots e_{n}]$ найдется такой базис $\{e^{'}_{2}, \dots e^{'}_{n}\}$, что $\{r^{'}_{2}e^{'}_{2}, \dots, r^{'}_{m}e^{'}_{m}\}$  ---  базис модуля $N'$ и $r^{'}_{i}$ делит $r^{'}_{i+1}$ ($i=2,\dots, m-1$)

Тогда \begin{equation*}
    u=r_{1}e_{1}+a^{'}_{2, u}e^{'}_{2}\cdot \cdot \cdot +a^{'}_{n, u}e^{'}_{n}
\end{equation*} для некоторых $a^{'}_{i, u}\in R;$ $i=2, \dots, n$.
Предположим, что $r_{1}$ делит не все элементы $a^{'}_{i, u}$. Пусть $a^{'}_{i, u}$ не делится на $r_{1}$. Тогда $a_{i, u}=xr_{1}+r$ для некоторого $x\in R$ и для некоторого $r\in R$ и $D\left({r}\right)<D\left({r_{1}}\right)$. Пусть $e^{'}_{1}=e_{1}+xe_{i}$. Рассмотрим новую систему элементов $\left\{{e^{'}_{1}, e_{2}, e_{i}\dots e_{n}}\right\}$ в модуле $M$. Легко видеть, что эта новая система является линейно независимой и порождает $M$, то есть является базисом. для модуля $M$ Заметим\begin{equation*}
    u=r_{1}e^{'}_{1}+a^{'}_{2, u}e^{'}_{2}+\cdots+\left({-x r_{1}+a^{'}_{i, u}}\right)e^{'}_{i}+\cdots+ a^{'}_{n, u}e^{'}_{n}.
\end{equation*} 
Но $-x r_{1}+a_{i, u}=r$ и евклидова норма $r$ меньше евклидовой нормы $r_{1}$. Это противоречит выбору $r_{1}$. Противоречие получено из предположения, что $r_{1}$ делит не все элементы $a^{'}_{i, u}$. Вывод $a^{'}_{i, u}=r_{1}a^{''}_{i, u}$ для некоторого $a^{''}_{i, u}\in R$ ($i=2,\dots,n$). Тогда \begin{equation*}
    u=r_{1}\left({e^{'}_{1}+a^{''}_{2, u}e^{'}_{2}+\cdots +a^{''}_{n, u}e^{'}_{n}}\right)
\end{equation*}\\ и пусть \begin{equation*}
    e^{''}_{1}=e^{'}_{1}+a^{''}_{2, u}e^{'}_{2}+\cdots +a^{''}_{n, u}e^{'}_{n}.
\end{equation*}

Рассмотрим новую систему элементов $\left\{{e^{''}_{1}, e^{'}_{2}, e^{'}_{i}\dots e^{'}_{n}}\right\}$ в модуле $M$. Легко видеть, что эта новая система является линейно независимой и порождает $M$, то есть является базисом для модуля $M$. Пусть $v\in N$ и тогда \begin{equation*}
    v=a^{'}_{1, v}e^{''}_{1}+a^{'}_{2, v}e^{'}_{2}\cdot \cdot \cdot +a^{'}_{m, v}e^{'}_{m}
\end{equation*} для некоторых $a^{'}_{i, u}\in R$ ($i=1, 2,\dots,n$)\\
Если бы $a^{'}_{1, v}$ не делилось на $r_{1}$ в кольце $R$, то $a^{'}_{1, u}=x'r_{1}+r'$ для некоторого $x' \in R$ и для некоторого $r'\in R$ и D$(r')<$  D$(r_{1})$, тогда $v-x'u\in N$ и $v-x'u = r'e''+w$, для некоторого $w\in N'$. Это противоречит выбору $r_{1}$. Следовательно $a'_{1,v}$ делится на $r_{1}$ в кольце $R$, то есть $a'_{1,v} = x'r_{1}$. Тогда $v-x'u\in N'=N\cap [e_{2}, \dots, e_{n}]$.
Получаем $\{r_{1}e^{'}_{1}, r^{'}_{2}e^{'}_{2}, \dots, r^{'}_{m}e^{'}_{m}\}$ — базис подмодуля $N$.

Повторяя приведенные выше рассуждения, убеждаемся, что $r^{'}_{2}$ делится на $r_{1}$. Предложение 4 доказано.
\end{proof}
\begin{proclaim}{Предложение 5}
  Если кольцо $K$ простой характеристики $p$ является конечно порожденным модулем над некоторым моногенным подкольцом $Z\left<{b}\right>$, то это кольцо $K$ финитно отделимо от подколец, содержащих $b$.
\end{proclaim}
\begin{proof}
 Пусть кольцо $K$ простой характеристики $p$ является конечно порожденным модулем над некоторым моногенным подкольцом $Z\left<{b}\right>$ и пусть $A$  ---  подкольцо и $a\not\in A$ и пусть $b$ некоторый элемент из $A$. Тогда $A$ является $Z\left<{b}\right>$--подмодулем $Z\left<{b}\right>$-модуля $K$. Обозначим через $\widetilde{a}$ образ элемента $a$ при каноническом гомоморфизме $Z\left<{b}\right>$--модулей: $K\rightarrow K/A$. По условию $\widetilde{a}\ne 0$ в модуле $K/A$ и этот модуль конечно порожден (как гомоморфный образ конечнопорожденного). Учитывая, что $K$  --- кольцо характеристики $p$, аддитивные группы $K$ и $K/A$ можно рассматривать как модули над кольцом многочленов $Z_{p}[x]$ с операцией $f\left({x}\right)* m=f\left({b}\right)\cdot m$ и эти модули конечно порождены. Как известно, конечно порожденный модуль над областью целостности является прямой суммой конечного набора циклических модулей (см. пп. 30). Отсюда вытекает, что для любого ненулевого элемента в кольце $K/A$, в частности $\widetilde{a}$, существует не нулевой многочлен $\varphi \left({x}\right)\in Z_{p}[x]$ на который этот элемент не делится, то есть \begin{equation*}
     \widetilde{a}\not\in \varphi\left({x}\right)* K/A.
 \end{equation*} Это означает, что в $K$ имеет место $a\not\in \varphi \left({x}\right)* K+A$. Заметим, что $\varphi \left({x}\right)* K$ является двусторонним идеалом в кольце $K$. Пусть $
     \gamma:K\rightarrow K/\left({\varphi \left({x}\right)* K}\right)$ канонический гомоморфизм колец. Из $a\not\in \varphi \left({x}\right)K+A$ вытекает $\gamma \left({a}\right)\not\in \gamma \left({A}\right)$. Осталось заметить, что кольцо $K/\left({\varphi \left({x}\right)* K}\right)$ конечно. Действительно, если $a_{1}, a_{2}\dots, a_{n}$ конечная система образующих модуля $Z\left<{b}\right>$--модуля $K$, то очевидно она является системой образующих $Z_{p}[x]$ модуля $K$. Откуда следует, что $\overline{a_{1}}, \overline{a_{2}}\dots, \overline{a_{n}}$  --- система образующих $Z_{p}[x]$ модуля $K/\left({\varphi \left({x}\right)* K}\right)$, где $\overline{a}_{i}$ образ элемента $a_{i}$ при каноническом гомоморфизме $Z_{p}[x]$--модулей: $K\rightarrow K/\left({\varphi \left({x}\right)* K}\right)$. $Z_{p}[x]$ модуль $K/\left({\varphi \left({x}\right)* K}\right)$ можно рассматривать как модуль над кольцом $Z_{p}[x]/I\left({\varphi \left({x}\right)}\right)$, где $I\left({\varphi \left({x}\right)}\right)$  --- главный идеал в кольце многочленов $Z_{p}[x]$, порожденный многочленом $\varphi \left({x}\right)$. Заметим, что фактор кольцо. $Z_{p}[x]/I\left({\varphi \left({x}\right)}\right)$ конечно. Но конечно порожденный модуль над конечным кольцом сам является конечным. Поэтому $K/\left({\varphi \left({x}\right)* K}\right)$  --- конечное множество. Предложение 5 доказано
\end{proof}
\begin{proclaim}{Предложение 6}
  Кольцо многочленов $Z_{p}[x]$ от одного переменного с коэффициентами из простого поля $Z_{p}$ финитно отделимо.
\end{proclaim}
\begin{proof}
Пусть $f\left({x}\right)\not\in A$, где $f\left({x}\right)$  --- некоторый многочлен из кольца $Z_{p}[x]$, а $A$ некоторое подкольцо кольца $Z_{p}[x]$. Если $A$ состоит многочленов нулевой степени, то $A$  --- конечное множество. Возможность финитного отделения любого элемента $f\left({x}\right)\not\in A$ от подкольца $A$ в этом случае следует из финитной аппроксимируемости конечно порожденных коммутативных колец, (см. \cite{kybl}).

Пусть в кольце $A$ есть некоторый многочлен $b=g\left({x}\right)$ и пусть степень многочлена $g\left({x}\right)$ равна $n$ и $n\ge 1$. Рассмотрим кольцо $Z_{p}[x]$ как модуль над моногенным подкольцом $Z<b>$.

Покажем, что этот модуль является конечно порожденным и множество $B=\{h\left({x}\right)\mid degh\left({x}\right)\le n\}$ (состоящее из всех многочленов $h\left({x}\right)$ из $Z_{p}[x]$ степени меньшей или равной $n$) является порождающим множеством этого модуля.

Пусть $[B]$  --- подмодуль модуля $Z_{p}[x]$ (как модуля над кольцом $Z<b>$). Покажем, что каждый многочлен $\alpha \left({x}\right)$ из $Z_{p}[x]$ принадлежит $[B]$.

Если $deg\alpha \left({x}\right)\le n$, то $\alpha \left({x}\right)\in B\subset [B]$ Пусть $deg\alpha \left({x}\right)>n$. Дальнейшие рассуждения проводим по индукции.

Предположим, что для любого многочленам $\beta \left({x}\right)$ (из $Z_{p}[x]$), степень которого меньше степени многочлена $\alpha \left({x}\right)$ имеет место $\beta \left({x}\right)\in [B]$.

Разделим многочлен $\alpha \left({x}\right)$ на многочлен $g\left({x}\right)$ с остатком, то есть верно равенство
\begin{equation*}
     \alpha \left({x}\right)=g\left({x}\right)\cdot \beta \left({x}\right)+r\left({x}\right)
\end{equation*}
 для некоторых многочленов $\beta \left({x}\right)$ и $r\left({x}\right)$ и $degr\left({x}\right)<degg\left({x}\right)$. Тогда по определению $r\left({x}\right)\in B\subset [B]$. Но из сравнения степеней многочленов следует, что \begin{equation*}
     deg\alpha \left({x}\right)=degg\left({x}\right)+deg\beta \left({x}\right).
 \end{equation*} Откуда \begin{equation*}
     deg\beta \left({x}\right)=deg\alpha \left({x}\right)-degg\left({x}\right)=deg\alpha \left({x}\right)-n<deg\alpha \left({x}\right).
 \end{equation*}По индукционному предположению $\beta \left({x}\right)\in [B]$. Но тогда $g\left({x}\right)\cdot \beta \left({x}\right)=b\cdot \beta \left({x}\right)\in b[B]\subset [B]$. Мы получили, что $g\left({x}\right)\cdot \beta \left({x}\right)\in [B]$ и $r\left({x}\right)\in [B]$. Поэтому\begin{equation*}
     \alpha \left({x}\right)=g\left({x}\right)\cdot \beta \left({x}\right)+r\left({x}\right)\in [B].
 \end{equation*}
Итак, мы доказали по индукции, что $Z_{p}[x]=[B]$. Заметим, что $B$  --- конечное множество. Это означает, что $Z_{p}[x]$ является конечно порожденным модулем над своим подкольцом $Z<b>$. Поскольку $b\in A$, то по предложению 5 существует конечное кольцо $F$ и гомоморфизм $\gamma:Z_{p}[x]\rightarrow F$ такой, что $\gamma \left({a}\right)\not\in \gamma \left({A}\right)$. Вывод: кольцо $Z_{p}[x]$ финитно отделимо. Предложение 6 доказано.
\end{proof}
\begin{corollary}
 Моногенное кольцо простой характеристики финитно отделимо.
\end{corollary}
\begin{proof}
Пусть $K=Z\left<{a}\right>$ и $pa=0$ для некоторого простого числа $p$. Тогда $Z\left<{a}\right>$ является гомоморфным образом кольца $Z^{* }_{p}[x]$ - многочленов из $Z_{p}[x]$ без подобных членов. По предложению 6 кольцо $Z_{p}[x]$ финитно отделимо. По лемме 1 заключаем, что $Z\left<{a}\right>$ финитно отделимо.
\end{proof}
\begin{lemma}
Если $a$  --- целокрутящийся элемент кольца $K$ и его целое кручение простое число, то моногенное кольцо $Z\left<{a}\right>$  --- финитно отделимо.
\end{lemma}
\begin{proof}

 Пусть $a$  --- целокрутящийся элемент кольца $K$ и его целое кручение $\tau _{a}=p$  --- простое число. Если $a$  --- целый алгебраический элемент, то по лемме 6 кольцо $Z\left<{a}\right>$ финитно отделимо.
Пусть $a$ не является целым алгебраическим. На основании определений существует унитарный многочлен $f\left({x}\right)$ с целыми коэффициентами такой, что $p\cdot f\left({a}\right)=0$. Пусть\begin{equation*}
    f\left({x}\right)=x^{n}+k_{1}x^{n-1}+\dots, k_{n-1}x.
\end{equation*}
Далее без ограничения общности будем считать, что кольцо $K$ совпадает с $Z\left<{a}\right>$.

Далее рассмотрим кольцо $Z^{* }[x]$  многочленов с целыми коэффициентами и без свободных членов и его фактор-кольцо по главному идеалу $I\left({pf\left({x}\right)}\right)$, порожденному элементом $pf\left({x}\right)$. Это фактор-кольцо $Z^{* }[x]/I\left({pf\left({x}\right)}\right)$ обозначим через $K^{* }$, а через $\overline{u\left({x}\right)}$  --- образ $u\left({x}\right)$ при каноническом гомоморфизме $Z^{* }[x]\rightarrow Z^{* }[x]/I\left({pf\left({x}\right)}\right)$, где $u\left({x}\right)\in Z^{* }[x]$. Покажем, что кольцо $K^{* }$ финитно отделимо. Рассмотрим следующие идеалы кольца $K^{* }$: главный идеал $I\left({\overline{f\left({x}\right)}}\right)$, порожденный элементом $\overline{f\left({x}\right)}$, и $p\cdot K^{* }$. Покажем, что\begin{equation*}
    I\left({\overline{f\left({x}\right)}}\right)\cap p\cdot K^{* }=\{0\}.
\end{equation*}  Предположим противное. Это означает, что
\begin{center}
    $\overline{f\left({x}\right)}\cdot \overline{g\left({x}\right)}=p\overline{h\left({x}\right)}$ или $\overline{f\left({x}\right)}=p\overline{h\left({x}\right)}$
\end{center} для некоторых многочленов $g\left({x}\right)$ и $h\left({x}\right)$ с целыми коэффициентами и без подобных членов, причем $\overline{f\left({x}\right)}\cdot \overline{g\left({x}\right)}\ne 0$ (*). Отсюда следует, что\begin{equation*}
    f\left({x}\right)g\left({x}\right)-ph\left({x}\right)\in I\left({pf\left({x}\right)}\right)
\end{equation*}  или \begin{equation*}
    f\left({x}\right)-ph\left({x}\right)\in I\left({pf\left({x}\right)}\right).
\end{equation*} Откуда получаем, что все коэффициенты многочлена $f\left({x}\right)g\left({x}\right)$ делятся на $p$ или все коэффициенты многочлена $f\left({x}\right)$ делятся на $p$. Последнее неверно, так как $f\left({x}\right)$  --- унитарный многочлен. Если все коэффициенты многочлена $f\left({x}\right)g\left({x}\right)$ делятся на $p$, то и все коэффициенты многочлена $g\left({x}\right)$ делятся на $p$ (это известное несложное упражнение). Тогда заключаем $f\left({x}\right)g\left({x}\right)\in I\left({pf\left({x}\right)}\right)$. Отсюда следует $\overline{f\left({x}\right)}\cdot \overline{g\left({x}\right)}=0$. Это противоречит (*). Противоречие доказывает, что $I\left({\overline{f\left({x}\right)}}\right)\cap p\cdot K^{* }=\{0\}$.

Откуда следует, что кольцо $K^{* }$ вкладывается в прямое произведение колец $K^{* }/I\left({\overline{f\left({x}\right)}}\right)$ и $K^{* }/p\cdot K^{* }$. В первом кольце имеет место соотношение $f\left({\overline{x}}\right)=0$ и оно является моногенным кольцом, порожденным элементом $\overline{x}$. По лемме 6 такое кольцо финитно отделимо. Кольцо $K^{* }/p\cdot K^{* }$ является моногенным кольцом простой характеристики. Оно является финитно отделимым по следствию предложению 6. Как показал А. И. Мальцев, прямое произведение конечного числа финитно отделимых колец финитно отделимо [1]. Поэтому кольцо $K^{* }$  финитно отделимо, как подкольцо финитно отделимого кольца  $K^{* }/I\left({\overline{f\left({x}\right)}}\right)\times K^{* }/p\cdot K^{* }$ Осталось заметить, что моногенное кольцо $Z\left<{a}\right>$ является гомоморфным образом кольца $K^{* }$ при каноническом гомоморфизме $u\left({x}\right)\rightarrow u\left({a}\right)$ и воспользоваться леммой 1. Лемма 7 доказана.
\end{proof}
\begin{lemma}
Если $a$  --- целокрутящийся элемент кольца $K$ и его целое кручение свободно от квадратов, то моногенное кольцо $Z\left<{a}\right>$  финитно отделимо.
\end{lemma}
\begin{proof}
 Пусть $a$  --- целокрутящийся элемент кольца $K$ и его целое кручение $\tau _{a}$ свободно от квадратов. Если $\tau _{a}=1$, то $Z\left<{a}\right>$  финитно отделимо по лемме 5. Пусть $\tau _{a}\ne 1$, тогда $\tau _{a}$ раскладывается в произведение различных простых чисел, то есть $\tau _{a}=p_{1}\cdots p_{n}$. Рассмотрим семейство натуральных чисел $k_{i}=\frac{\tau _{a}}{p_{i}}$ ($i=1, \dots n$) Из определений вытекает, что Н.О.Д $\left({k_{1}, \dots k_{n}}\right)=1$. По теореме Евклида о линейном выражении Н.О.Д существует семейство целых чисел $z_{i}$ ($i=1, \dots, n$) такое, что
\begin{equation}
    z_{1}k_{1}+\dots +z_{n}k_{n}=1 \tag{**}.
\end{equation}
Рассмотрим семейство идеалов $I_{k_{i}}$ кольца $K$ (см. пп. 11). Из равенства ($* * $) следует, что $\bigcap\limits_{i=1}^{n}{I_{k_{i}}}=0$. Действительно, если $u\in \bigcap\limits_{i=1}^{n}{I_{k_{i}}}$, то\\ $k_{i}u=0, где i=1, \dots, n$, откуда\begin{equation*}
    u=1\cdot u=\left({z_{1}k_{1}+\dots +z_{n}k_{n}}\right)\cdot u=0.
\end{equation*}
Отсюда следует, что кольцо $K$ вкладывается в прямое произведение колец $K/I_{k_{1}}\times\dots\times K/I_{k_{n}}$. Далее, без ограничения общности, будем считать, что кольцо $K$ совпадает с моногенным кольцом $Z\left<{a}\right>$. Покажем, что каждое кольцо $K/I_{k_{i}}$  финитно отделимо. Действительно, $\tau _{a}f\left({a}\right)=0 $ для некоторого унитарного многочлена $f\left({x}\right)$ с целыми коэффициентами. Тогда в кольце $K/I_{k_{i}}$ имеет место $p_{i}f\left({\overline{a}}\right)=0$, где $\overline{a}$  --- образ $a$ при каноническом гомоморфизме $K\rightarrow K/I_{k_{i}}$.

Это означает, что $\overline{a}$  --- целокрутящийся элемент кольца $K/I_{k_{i}}$ и его целое кручение простое число $p_{i}$. По лемме 7 моногенное кольцо $Z\left<{\overline{a}}\right>$   финитно отделимо, то есть $ K/I_{k_{i}}=Z\left<{\overline{a}}\right>$ финитно отделимо, Как показал А. И. Мальцев, прямое произведение конечного числа финитно отделимых колец финитно отделимо. Поэтому кольцо $K$  финитно отделимо, как подкольцо финитно отделимого кольца $K/I_{k_{1}}\times\dots\times  K/I_{k_{n}}$. Лемма 8 доказана.
\end{proof}
\begin{lemma}
Пусть для некоторого элемента $a$ кольца $K$ выполняется равенство $l_{0}a^{m}+l_{1}a^{m-1}+\dots +l_{m-1}a=0$ для некоторого натурального числа $m$ и некоторых целых чисел $l_{0}, l_{1}, \dots, l_{m-1}$, причем $l_{0}>0$ и Н.О.Д $\left({l_{0}, l_{1}, \dots, l_{m-1}}\right)=1, $ а также равенство $df\left({a}\right)=0$ для некоторого унитарного многочлена с целыми коэффициентами $f\left({x}\right)$ (без восвободного члена), и некоторого натурального числа $d>1$ Тогда в кольце $K$ выполнено равенство $\varphi \left({a}\right)=0$ для некоторого унитарного многочлена $\varphi \left({x}\right)$ степени $\le $ $m$ с целыми коэффициентами (без свободного члена).
\end{lemma}
\begin{proof}
 1) Без ограничения общности, по предложению 1 будем считать, что $df\left({x}\right)$  --- минимальный многочлен элемента $a$ в кольце $K$. Пусть $d=p$  --- простое число, то есть $pf\left({a}\right)=0$ Рассмотрим многочлен \begin{equation*}
     g\left({x}\right)=l_{0}x^{m}+l_{1}x^{m-1}+\dots + l_{m-1}x.
 \end{equation*} Заметим, что не все коэффициенты многочлена $g\left({x}\right)$ делятся на $p$. Далее считаем, что среди многочленов $\widetilde{g}\left({x}\right)$ степени $\le m$, без свободного члена, у которых не все коэффициенты делятся на $p$, и для которых $\widetilde{g}\left({a}\right)=0$, выбран многочлен наименьшей степени. Пусть $\widetilde{m}$  --- степень многочлена $\widetilde{g}\left({x}\right)$,  $\widetilde{\widetilde{l}_{0}}$  --- старший коэффициент многочлена $\widetilde{g}\left({x}\right)$, а $n$  --- степень многочлена $f\left({x}\right)$. По определению минимального многочлена $n\le \widetilde{m}$ Предположим, что $\widetilde{\widetilde{l}_{0}}$ делится на $p$, то есть $\widetilde{\widetilde{l}_{0}}=tp$ для некоторого целого числа $t$. Пусть \begin{equation*}
     h\left({x}\right)=\widetilde{g}\left({x}\right)-tpf\left({x}\right)x^{\widetilde{m}-n}.
 \end{equation*}
Заметим, что $h\left({a}\right)=0$ и не все коэффициенты $h\left({x}\right)$ делятся на $p$. В противном случае все коэффициенты многочлена $\widetilde{g}\left({x}\right)$ делятся на $p$. Но степень многочлена $h\left({x}\right)$ меньше $\widetilde{m}$. Это противоречит выбору $\widetilde{m}$.\\
Вывод: $\widetilde{\widetilde{l}_{0}}$ не делится на $p$. Тогда Н.О.Д $\left({\widetilde{\widetilde{l}_{0}}, p}\right)=1$. По теореме Евклида $t\widetilde{\widetilde{l}_{0}}+sp=1$ для некоторых целых $t, s$. Пусть\begin{equation*}
    \varphi \left({x}\right)=t\cdot g\left({x}\right)+s\cdot pf\left({x}\right)x^{\widetilde{m}-n}.
\end{equation*} Тогда $\varphi \left({x}\right)$ унитарный многочлен степени $\widetilde{m}$ с целыми коэффициентами. и $\varphi \left({a}\right)=0$. Вывод: утверждение леммы 9 верно для этого случая.

2) Пусть для всех чисел $d'<d$ утверждение леммы 9 верно при условии $d'f'\left({a}\right)=0$ для некоторого унитарного многочлена с целыми коэффициентами $f'\left({x}\right)$. То есть $\varphi '\left({a}\right)=0$ для некоторого унитарного многочлена $\varphi '\left({x}\right)$ степени $\le $ $m$ с целыми коэффициентами.

3) Пусть $d$  --- не простое число. Тогда $d=p\cdot q$, где для некоторого простого $p$ и целого числа $q>1$. Пусть $I_{q}=\{b\in Z\left<{a}\right>| qb=0\}$  --- идеал кручения (см. пп. 11). Тогда в фактор--кольце $Z\left<{a}\right>/I_{q}$ имеет место \begin{center}
    $pf\left({\overline{a}}\right)=0$ и $g\left({\overline{a}}\right)=0$ в кольце $Z\left<{a}\right>/I_{q}$.
\end{center}  Это означает, что по первому случаю $\varphi \left({\overline{a}}\right)=0$ для некоторого унитарного многочлена $\varphi \left({x}\right)$ степени меньшей или равной $m$ с целыми коэффициентами. Откуда $q\varphi \left({a}\right)=0$ в кольце $Z\left<{a}\right>$. Тогда по предположению пункта 2, поскольку $q<d$, $\varphi '\left({a}\right)=0$ для некоторого унитарного многочлена $\varphi '\left({x}\right)$ степени $\le $ $m$ с целыми коэффициентами.
Лемма 9 доказана.
\end{proof}
\begin{lemma}
Если для некоторого элемента $a$ кольца $K$ выполняется равенство: $z_{0}a^{m}+z_{1}a^{m-1}+\dots+z_{m-1}a=0$ для некоторого натурального числа $m$ и некоторых целых чисел $z_{0}, z_{1}, \dots, z_{m-1}$, причем $z_{0}>0$ и \\Н.О.Д $\left({z_{0}, z_{1}, \dots, z_{m-1}}\right)=k, $ а также равенство $df\left({a}\right)=0$ для некоторого унитарного многочлена с целыми коэффициентами $f\left({x}\right)$ (без свободного члена), и некоторого натурального числа $d>1$, тогда в кольце $K$ выполнено равенство $k\varphi \left({a}\right)=0$ для некоторого унитарного многочлена $\varphi \left({x}\right)$ тепени $\le $ $m$ с целыми коэффициентами (без свободного члена).
\end{lemma}
\begin{proof}
Пусть $I_{k}=\{b\in K\mid kb=0\}$  --- идеал кручения (см. пп. 13) В фактор--кольце $K/I_{k}$ выполнено условие леммы 9 для элемента $\overline{a}$  --- образа $a$ при каноническом гомоморфизме $K\rightarrow K/I_{k}$. Применяя лемму 9 получим, что в кольце $K/I_{k}$ равенство $\varphi \left({\overline{a}}\right)=0$ для некоторого унитарного многочлена $\varphi \left({x}\right)$ степени $\le m$ с целыми коэффициентами (без восвободного члена). Тогда в кольце $K$ выполнено $k\varphi \left({a}\right)=0$. Лемма 10 доказана.
\end{proof}

\section{Основные результаты}
\setcounter{theorem}{0}
Основным результатом данной работы является следующая теорема.
\begin{theorem}
1. Для того, чтобы моногенное кольцо $Z\left<{\left. {a}\right>}\right. $ было финитно отделимым необходимо и достаточно, чтобы для некоторых целых чисел $k, n, k_{1}, k_{2},\dots, k_{n-1}$ ($n>0$) выполнялось равенство\begin{equation*}
    k\left({a^{n}+k_{1}a^{n-1}+k_{2}a^{n-2}+k_{n-1}a}\right)=0,
\end{equation*} причем $k$ ---  положительное целое число свободное от квадратов (то есть $k$  --- произведение различных простых чисел или $k=1$ ). \\
2. Пусть моногенное кольцо $Z\left<{\left. {a}\right>}\right. $ задано конечным набором определяющих соотношений $f_{1}\left({a}\right)=0, f_{2}\left({a}\right)=0, \dots, f_{m}\left({a}\right)=0$, где $f_{i}$  --- многочлены с целыми коэффициентами без свободных членов. Для того, чтобы моногенное кольцо $Z\left<{\left. {a}\right>}\right. $ было финитно отделимым необходимо и достаточно, чтобы выполнялось два условия:\\
(i) Н.О.Д коэффициентов всех многочленов $f_{i}$ (вместе взятых) был свободен от квадратов;\\
(ii) унитарный многочлен, являющийся Н.О.Д'ом всех многочленов $f_{i}$ над полем рациональных чисел, имел целые коэффициенты.

\end{theorem}
\begin{proof}
Утверждение 1 настоящей теоремы есть совокупность лемм 4 и 8. Докажем утверждение 2.

Пусть моногенное кольцо $Z\left<{\left. {a}\right>}\right. $ задано конечным набором определяющих соотношений $f_{1}\left({a}\right)=0, f_{2}\left({a}\right)=0, \dots, f_{m}\left({a}\right)=0$, где $f_{i}$  --- многочлены с целыми коэффициентами и без подобных членов. Пусть $V=\left<{\left. {f_{1}\left({x}\right), \dots f_{m}\left({x}\right)}\right>}\right. $  ---  идеал в кольце многочленов $Z[x]$, порожденный множеством $f_{1}\left({x}\right), \dots, f_{m}\left({x}\right)$. \\
а) Предположим, что $Z\left<{\left. {a}\right>}\right. $  --- финитно отделимое кольцо. Тогда по лемме 4 в кольце $Z\left<{\left. {a}\right>}\right. $ имеет место равенство $k\cdot \varphi \left({a}\right)=0$ для некоторого унитарного многочлена $\varphi \left({x}\right)$ с целыми коэффициентами (без свободного члена) и некоторого натурального числа $k$ свободного от квадратов.

Покажем, что выполняются условия (i) и (ii) настоящей теоремы. Если допустить, что (i) не выполнено, то это бы означало, что все коэффициенты всех многочленов $f_{i}\left({x}\right)$ ($i=1, \dots, m $) делятся на квадрат $p^{2}$ некоторого простого числа $p$. Отсюда бы вытекало, что это свойство выполнялось бы для всех многочленов идеала $V=\left<{\left. {f_{1}\left({x}\right), \dots f_{m}\left({x}\right)}\right>}\right. $. Этот идеал состоит в точности из всех многочленов $g\left({x}\right)$ с целыми коэффициентами без свободных членов, для которых в кольце $Z\left<{\left. {a}\right>}\right. $ выполнено равенство $g\left({a}\right)=0$. Поэтому $k\cdot\varphi\left({x}\right) \in V$. Отсюда следует, что старший коэффициент многочлена $k\cdot \varphi \left({x}\right)$ делится на $p^{2}$, то есть число $k$ делится на $p^{2}$. Противоречие. Вывод: условие (i) выполнено.

Пусть унитарный многочлен $\gamma \left({x}\right)$  ---  Н.О.Д $f_{1}\left({x}\right), \dots f_{m}\left({x}\right)$ над полем рациональных чисел. Пусть \begin{equation*}
    \gamma \left({x}\right)=x^{n}+\frac{p_{1}}{q_{1}}x^{n-1}+\dots \frac{p_{n}}{q_{n}}.
\end{equation*} Мы используем тот факт, что кольцо многочленов с одной переменной над полем является кольцом главных идеалов. То есть у каждого набора элементов есть Н.О.Д. Из определений вытекает, что \begin{equation*}
    g_{1}\left({x}\right){f_{1}\left({x}\right)+\dots+ g_{m}\left({x}\right) f_{m}\left({x}\right)}=\gamma \left({x}\right)
\end{equation*} для некоторых многочленов $g_{1}\left({x}\right), \dots g_{m}\left({x}\right)$ с рациональными коэффициентами. Пусть $l$  ---  Н.О.К всех знаменателей коэффициентов многочленов $g_{1}\left({x}\right), \dots g_{m}\left({x}\right)$. Тогда $l\cdot g_{1}\left({x}\right)\dots l\cdot g_{m}\left({x}\right)$  --- многочлены с целыми коэффициентами и \begin{equation*}
    \left({l\cdot g_{1}\left({x}\right)}\right) {f_{1}\left({x}\right)+\dots +\left({ l\cdot g_{m}}\left({x}\right)\right) f_{m}\left({x}\right)}=l\cdot \gamma \left({x}\right).
\end{equation*} Откуда заключаем, что $l\cdot \gamma \left({x}\right)\in V$. По предложению 1 минимальный многочлен элемента $a$ в кольце $K$ имеет вид $d\cdot f\left({x}\right)$, то есть в $K$ выполнено $d\cdot f\left({a}\right)=0$, где $f\left({x}\right)$  --- некоторый унитарный многочлен с целыми коэффициентами и без свободного члена, степень которого $n$ равна алгебраической степени элемента $a$, поскольку $d\cdot f\left({x}\right)\in V$. Поскольку $\left({l\cdot \gamma }\right)\left({a}\right)=0$, то степень многочлена $l\gamma \left({x}\right)$ $\ge $ $n$. Поскольку в кольце многочленов на полем рациональных чисел многочлен $d\cdot f\left({x}\right)$ (как и любой элемент из $V$) делится на $\gamma \left({x}\right)$, заключаем, что степень многочлена $\gamma \left({x}\right)$ $\le $ $n$. Вывод: степень многочлена $\gamma \left({x}\right)$ равна $n$. Полученное означает, что $d\cdot f\left({x}\right)=\alpha \cdot \gamma \left({x}\right)$ для некоторого рационального числа $\alpha $. Если равны многочлены, то равны их старшие коэффициенты, откуда $d=\alpha $. Поэтому $f\left({x}\right)=\gamma \left({x}\right)$, то есть $\gamma \left({x}\right)$  --- многочлен с целыми коэффициентами. Вывод: условие (ii) выполнено. \\
б) Предположим, что выполнено условие (i) и (ii).

Пусть $\gamma \left({x}\right)$  --- Н.О.Д $f_{1}\left({x}\right), \dots f_{m}\left({x}\right)$ над полем рациональных чисел Из условия (ii) следует, что $\gamma \left({x}\right)$  --- многочлен с целыми коэффициентами, Как уже отмечалось выше, $l\cdot \gamma \left({x}\right)\in V$ для некоторого натурального числа $l$. Это означает, что $\left({l\cdot \gamma }\right)\left({a}\right)=0$. Пусть $n_{i}$  --- степень многочлена $f_{i}\left({x}\right)$.

Пусть \begin{equation*}
    g\left({x}\right)=f_{1}\left({x}\right)+f_{2}\left({x}\right)x^{n_{1}}+f_{3}\left({x}\right)x^{n_{1}+n_{2}}+\dots f_{m}\left({x}\right)x^{n_{1}+n_{2}+\dots n_{m-1}}
\end{equation*}
Тогда $k$  --- Н.О.Д коэффициентов всех многочленов $f_{i}$ (вместе взятых) совпадает с Н.О.Д коэффициентов многочлена $g\left({x}\right)$. По условию (i) $k$ свободен от квадратов; Из определений следует, что $g\left({x}\right)\in V$, то есть $g\left({a}\right)=0$. По лемме 10 в кольце $K$ выполнено равенство $k\varphi \left({a}\right)=0$ для некоторого унитарного многочлена $\varphi \left({x}\right)$ с целыми коэффициентами (без свободных членов). По лемме 8 кольцо $Z\left<{a}\right>$ финитно отделимо.

Теорема доказана полностью.
\end{proof}
\begin{corollary}
 Конечно порожденное PI-кольцо без кручения является финитно отделимым в том и только в том случае, когда его аддитивная группа конечно порождена.
\end{corollary}
\begin{proof}
 Пусть $K$  ---  конечно порожденное PI-кольцо без кручения. По теореме Ширшова о высоте найдутся натуральное число $h$ и элементы $a_{1}, a_{2}, \dots, a_{n}\in K$ такие, что любой элемент кольца $K$ представляется в виде линейной комбинации элементов вида\begin{equation*}
      w = a^{\alpha _{1}}_{i_{1}}a^{\alpha _{2}}_{i_{2}}\dots, a^{\alpha _{k}}_{i_{k}},
 \end{equation*} для некоторого натурального числа $k<h$ (число $k$ называют высотой слова $w$ над $a_{1}, a_{2}, \dots, a_{n}$).

Предположим, что $K$  --- финитно отделимое кольцо. По лемме 2 $d_{j}\cdot f_{j}\left({a_{j}}\right)=0$ для некоторого положительного целого числа $d_{i}$ и некоторого унитарного многочлена $f_{j}\left({x}\right)$ с целыми коэффициентами и без свободного члена. Поскольку $K$  --- кольцо без кручения, то можно заключить, что $f_{j}\left({a_{j}}\right)=0$, откуда следует, что $a^{\alpha _{j}}_{i_{j}}$ линейно выражается через $a^{\beta _{j}}_{i_{j}}$ где $\beta _{j}<n_{j}$ и $n_{j}$  --- степень многочлена $f_{j}\left({x}\right)$. Это означает, что любой элемент кольца $K$ представляется в виде линейной комбинации элементов вида \begin{equation*}
    w = a^{\beta _{1}}_{i_{1}}a^{\beta _{2}}_{i_{2}}\dots, a^{\beta _{k}}_{i_{k}},
\end{equation*} где $\beta _{j}<m$ и $m=\max\{n_{j} \mid j=1, 2\dots k\}$, $k<h$ Элементов $w$, указанного выше вида, конечное множество. Вывод: аддитивная группа кольца $K$ конечно порождена. Обратное утверждение есть лемма 5. Следствие теоремы доказано.
\end{proof}

\section{Заключение}

Полученные в настоящей работе результаты показывают, что ситуация со свойством финитной отделимости в кольцах значительно сложнее, чем в группах и полугруппах, даже для коммутативного случая. Из теоремы о разложении конечно порожденных абелевых групп в прямую сумму циклических групп следует, что любая конечно порожденная коммутативная группа  ---  финитно отделима. Алгоритмически проверяемые необходимые и достаточные условия финитной отделимости конечно порожденных коммутативных полугрупп описаны в работе автора в 1983 (см. [7]).  Вопрос же о необходимых и достаточных условиях финитной отделимости конечно порожденных коммутативных колец в общем виде открыт до сих пор. На трудности решения этой проблемы указывает тот факт, что, как показано в настоящей работе, кольцо многочленов от одной переменной над простым полем является финитно отделимым, а от двух и более переменных  ---  уже не является. И еще тот факт, что для моногенных колец решение этой проблемы оказалось нетривиальным. Одним из приложений полученного здесь решения является возможность алгоритмического решения проблемы вхождения элемента в подкольцо (не обязательно конечно порожденное) для финитно отделимых моногенных колец. А это --- нетривиальная диофантова задача в кольце многочленов с целыми коэффициентами.

\end{document}